\documentclass[11pt,twoside]{article}
\usepackage{appendix}
\usepackage[all]{xy}
\usepackage{stmaryrd}

\usepackage{ifxetex}
\usepackage{textpos}
\usepackage{natbib}
\usepackage{kpfonts}
\usepackage[a4paper,hmargin=2.8cm,vmargin=2.0cm,includeheadfoot]{geometry}
\usepackage{ifxetex}
\usepackage{stackengine}
\usepackage{tabularx,longtable,multirow,subfigure,caption}
\usepackage{fancyhdr}
\usepackage{color}
\usepackage[tight,ugly]{units}
\usepackage{url}
\usepackage{float}
\usepackage[english]{babel}
\usepackage{amsmath}
\usepackage{graphicx}
\usepackage[colorinlistoftodos]{todonotes}
\usepackage{dsfont}
\usepackage{epstopdf} 
\usepackage{natbib}
\usepackage{array}
\usepackage{latexsym}
\usepackage{etoolbox}

\usepackage{appendix}
\usepackage[all]{xy}
\usepackage{stmaryrd}

\usepackage{amsthm}
\theoremstyle{plain}

\usepackage{url}
\usepackage{hyperref}
\hypersetup{colorlinks=true,citecolor=green,linkcolor=blue,urlcolor=magenta}
\usepackage[capitalize,nameinlink,noabbrev]{cleveref}

\newtheorem{thm}{Theorem}[section]
\newtheorem{prop}[thm]{Proposition}
\newtheorem{cor}[thm]{Corollary}
\newtheorem{lem}[thm]{Lemma}

\theoremstyle{definition}
\newtheorem{defn}[thm]{Definition}
\newtheorem{rmk}[thm]{Remark}
\newtheorem{ex}[thm]{Example}

\usepackage{algorithm} 
\usepackage{algpseudocode} 
\algrenewcommand\algorithmicrequire{\textbf{Input:}}
\algrenewcommand\algorithmicensure{\textbf{Output:}}

\makeatletter
\newcommand*\bigcdot{\mathpalette\bigcdot@{.5}}
\newcommand*\bigcdot@[2]{\mathbin{\vcenter{\hbox{\scalebox{#2}{$\m@th#1\bullet$}}}}}
\makeatother

\usepackage{tcolorbox}

\usepackage{import}
\usepackage{xifthen}
\usepackage{pdfpages}
\usepackage{transparent}

\usepackage{lipsum}

\newcommand\blfootnote[1]{%
  \begingroup
  \renewcommand\thefootnote{}\footnote{#1}%
  \addtocounter{footnote}{-1}%
  \endgroup
}

\ifxetex
\else
\fi

\def\@makechapterhead#1{%
  \vspace*{10\p@}%
  {\parindent \z@ \raggedright 
    \interlinepenalty\@M
    \Huge \bfseries 
    \thechapter \space\space #1\par\nobreak
    \vskip 30\p@
  }}

\def\@makeschapterhead#1{%
  \vspace*{10\p@}%
  {\parindent \z@ \raggedright
    \sffamily
    \interlinepenalty\@M
    \Huge \bfseries  
    #1\par\nobreak
    \vskip 30\p@
  }}

\def\Beginboxit
   {\par
    \vbox\bgroup
	   \hrule
	   \hbox\bgroup
		  \vrule \kern1.2pt %
		  \vbox\bgroup\kern1.2pt
   }

\def\Endboxit{%
			      \kern1.2pt
		       \egroup
		  \kern1.2pt\vrule
		\egroup
	   \hrule
	 \egroup
   }

\newenvironment{boxit*}{\Beginboxit\hbox to\hsize{}}{\Endboxit}

\allowdisplaybreaks

\makeatletter
\newcounter{elimination@steps}
\newcolumntype{R}[1]{>{\raggedleft\arraybackslash$}p{#1}<{$}}
\def\elimination@num@rights{}
\def\elimination@num@variables{}
\def\elimination@col@width{}

\newcommand{\eliminationstep}[2]
{
    \ifnum\value{elimination@steps}>0\leadsto\quad\fi
    \left[
        \ifnum\elimination@num@rights>0
            \begin{array}
            {@{}*{\elimination@num@variables}{R{\elimination@col@width}}
            |@{}*{\elimination@num@rights}{R{\elimination@col@width}}}
        \else
            \begin{array}
            {@{}*{\elimination@num@variables}{R{\elimination@col@width}}}
        \fi
            #1
        \end{array}
    \right]
    & 
    \begin{array}{l}
        #2
    \end{array}
    &
    \addtocounter{elimination@steps}{1}
}
\makeatother

\makeatletter  
\def\colvec#1{\expandafter\colvec@i#1,,,,,,,,,\@nil}
\def\colvec@i#1,#2,#3,#4,#5,#6,#7,#8,#9\@nil{%
  \ifx$#2$ \begin{bmatrix}#1\end{bmatrix} \else
    \ifx$#3$ \begin{bmatrix}#1\\#2\end{bmatrix} \else
      \ifx$#4$ \begin{bmatrix}#1\\#2\\#3\end{bmatrix}\else
        \ifx$#5$ \begin{bmatrix}#1\\#2\\#3\\#4\end{bmatrix}\else
          \ifx$#6$ \begin{bmatrix}#1\\#2\\#3\\#4\\#5\end{bmatrix}\else
            \ifx$#7$ \begin{bmatrix}#1\\#2\\#3\\#4\\#5\\#6\end{bmatrix}\else
              \ifx$#8$ \begin{bmatrix}#1\\#2\\#3\\#4\\#5\\#6\\#7\end{bmatrix}\else
                 \PackageError{Column Vector}{The vector you tried to write is too big, use bmatrix instead}{Try using the bmatrix environment}
              \fi
            \fi
          \fi
        \fi
      \fi
    \fi
  \fi 
}  
\makeatother

\robustify{\colvec}

\newcommand{\R}[0]{\mathds{R}} 
\newcommand{\Z}[0]{\mathds{Z}} 
\newcommand{\N}[0]{\mathds{N}} 
\newcommand{\Q}[0]{\mathds{Q}} 
\ifxetex
\newcommand{\C}[0]{\mathds{C}} 
\else
\newcommand{\C}[0]{\mathds{C}} 
\fi
\renewcommand{\d}[0]{\mathrm{d}} 

\ifxetex
\else

\fi

\definecolor{darkgreen}{rgb}{0,0.6,0}

\newcommand{\blue}[1]{{\color{blue}#1}}

\renewcommand{\emph}[1]{\blue{\bf{#1}}}

\title{Galois orbits of torsion points over polytopes near atoral sets}
\author{Chenying Lin}

\IfFileExists{coursework_backup.aux}{
	\IfFileExists{coursework.aux}{}{
		\message{We saw a default backup.aux file, let's use it instead of the main aux file.}
		\makeatletter
		\input{coursework_backup.aux}
		\makeatother
	}
}{}

\begin{document}

\maketitle

\begin{abstract}
    Given an essentially atoral Laurent polynomial $P$, we show an equidistribution theorem for the function $\operatorname{log}|P|$ on specific subsets of Galois orbits of torsion points of the $d$-dimensional algebraic torus $\mathbb{G}^d_m(\overline{\Q})$. The specific subsets under consideration are the preimages of $d$-dimensional polytopes within the hypercube $[0,1]^d$ under the cotropicalization map. This generalises an equidistribution theorem of V. Dimitrov and P. Habegger, who considered  only all Galois orbits that correspond to the entire hypercube $[0,1]^d$. In addition, we provide an estimate for the convergence speed of this equidistribution, expressed as a negative power of the strictness degree. Our approach is to derive an alternative version of Koksma's inequality over polytopes. 
    
    As an application, we provide the convergence speed of heights on a sequence of projective points for a specific two-dimensional example, answering a question posed by R. Gualdi and M. Sombra. In the appendix, we present an algorithm to compute the explicit value of the power of the strictness degree.
\end{abstract}

\blfootnote{


\ \ The author is supported by the collaborative research center SFB 1085 \textit{Higher Invariants} funded by the DFG and Project PID2023-147642NB-100 of the Universitat de Barcelona.
}

\setcounter{tocdepth}{1}

\tableofcontents

\newpage

\section{Introduction}

Let $\overline{\Q}$ be a fixed algebraic closure of $\Q$. Let $d\geq 1$ be an integer and let $\mathbb{G}_m^d$ denote the $d$-dimensional multiplicative group variety over $\overline{\Q}$, where $\mathbb{G}_m$ is isomorphic to the group variety $\operatorname{Spec}\overline{\Q}[T^{\pm 1}]$. Let $\iota:\overline{\Q}\hookrightarrow\C$ be a fixed embedding; let also denote by $\iota$ the induced map $(\overline{\Q}^*)^d\hookrightarrow(\C^*)^d$. We identify
\begin{itemize}

    \item $\mathbb{G}_m^d$ with its $\overline{\Q}$-points $\mathbb{G}_m^d(\overline{\Q})=(\overline{\Q}^*)^d$,
    \item $\iota(x)$ with $x$ for every $x\in\overline{\Q}$ and $\iota(\omega)$ with $\omega$ for every $\omega\in\mathbb{G}_m^d$.
\end{itemize}

Let $\omega\in\mathbb{G}_m^d$ be a torsion point. We define the \textit{strictness degree} of $\omega$ to be
\[
\delta(\omega)\coloneqq\operatorname{inf}\big\{ |a|:a\in\Z^d\backslash\{0\}\text{, }\omega^a=1 \big\},
\]
where $\omega^a=\omega_1^{a_1}\cdot\cdots\cdot\omega_d^{a_d}$ for $\omega=(\omega_1,\ldots,\omega_d)$ and $a=(a_1,\ldots,a_d)$, and $|\cdot|$ is the maximum norm. By \cite[Chapter 3]{Bombieri_Gubler_2006}, an algebraic subgroup of $\mathbb{G}_m^d$ of dimension $d-1$ takes the form $H_a=\{\omega\in\mathbb{G}_m^d:\omega^a=1\}$
for an $a\in\Z^d\backslash\{0\}$. Moreover, every proper algebraic subgroup of $\mathbb{G}_m^d$ is included in $H_a$ for some $a$. Notice that $\omega$ is not contained in $H_a$ for finitely many $a$ with $|a|<\delta(\omega)$. Hence, when $\delta(\omega)$ becomes larger, in this way we can easily find more and more algebraic subgroups of $\mathbb{G}_m^d$ that does not contain the Galois orbit of $\omega$.

The well-known Bilu's equidistribution theorem for the Galois orbits of points of small heights \cite{bilu1997limit} implies that when $\delta(\omega)$ becomes arbitrarily large, the average of the Dirac measure at $\omega^\sigma$ over $\sigma\in\operatorname{Gal}(\Q(\omega)/\Q)$ weakly converges to the normalized Haar measure on $(S^1)^d\coloneqq\{(z_1,\ldots,z_d)\in(\C^*)^d:|z_1|=\ldots=|z_d|=1\}$. In other words, for a continuous function $f:(\C^*)^d\to\R$ with compact support, we have
\begin{equation}\label{eq: Bilu's thm}
    \frac{1}{[\Q(\omega):\Q]}\sum_{\sigma\in\operatorname{Gal}(\Q(\omega)/\Q)}f(\omega^{\sigma})\longrightarrow\int_{[0,1)^d}f(e^{\mathrm{i} 2\pi x_1},\ldots,e^{\mathrm{i}2\pi x_d})\d x_1\cdots\d x_d
\end{equation}
as $\delta(\omega)\to\infty$.

In \cite{DimitrovHabegger}, V. Dimitrov and P. Habegger proved a quantitative version of the equidistribution result for functions $f=\operatorname{log}|P|$ where $P\in\overline{\Q}[T_1^{\pm1},\ldots,T_d^{\pm1}]$ is an \textit{essentially atoral} Laurant polynomial. 
For the definition of essentially atoral polynomial, we refer to \Cref{Notation Section}. Note that the function $f=\operatorname{log}|P|$ has singularities at the vanishing set of $P$. Therefore, the result of Dimitrov and Habegger is an extension of (\ref{eq: Bilu's thm}). In this case, the integration in the right hand side of (\ref{eq: Bilu's thm}) becomes the \textit{(logarithmic) Mahler measure}
\[
m(P)\coloneqq\int_{[0,1)^d}\operatorname{log}\big|P(e^{\mathrm{i}2\pi x_1},\ldots,e^{\mathrm{i}2\pi x_d})\big|\d x_1\cdots\d x_d
\]
of the polynomial $P$. The study of Mahler measure plays an important role in height theory, which studies the arithmetic complexity of points and cycles in varieties.

Note that given any torsion point $\omega=(\omega_1,\ldots,\omega_d)\in\mathbb{G}_m^d$, we have $\omega_j=e^{\mathrm{i}2\pi x_j}$ with a real number $x_j\in[0,1)$ for all $j$. The main result of V. Dimitrov and P. Habegger, as presented in \cite[Theorem 1.1]{DimitrovHabegger}, considers the sum over all Galois orbits of $\omega$, i.e. the whole higher dimensional cube $[0,1)^d$. However, in certain applications, it becomes necessary to focus on the sum restricted to specific subsets of the full Galois orbits, i.e. a subset of $[0,1)^d$. Here we consider subsets obtained by partitioning the unit hypercube into polyhedral subsets. The following result is the main theorem of this article, generalizing \cite[Theorem 1.1]{DimitrovHabegger}.

Denote $e(x)=(e^{\mathrm{i}2\pi x_1},\ldots,e^{\mathrm{i}2\pi x_d})$ for $x=(x_1,\ldots,x_d)\in\R^d$. We define the \textit{cotropicalization map} by
\[
\operatorname{ct}:(\C^*)^d\longrightarrow [0,1)^d, \quad (z_1,\ldots,z_d) \longmapsto (x_1,\ldots,x_d),
\]
where $z_i=r_ie^{\mathrm{i}2\pi x_i}$ for $r_i>0$ and $x_i\in[0,1)$. 

\begin{thm}\label{main thm}
    Let $d,k$ be integers and let $P\in\overline{\Q}[T_1^{\pm 1},\ldots,T_d^{\pm 1}]\backslash\{0\}$ be essentially atoral with at most $k$ nonzero terms. Let $\Delta\subset[0,1)^d$ be a polytope of dimension $d$. Then there exists a constant $\kappa=\kappa(d,k)>0$ such that the following holds.

    Given any torsion point $\omega\in\mathbb{G}_m^d$ suppose that the strictness degree $\delta(\omega)$ is large enough. Then $P(\omega^\sigma)\neq 0$ for all $\sigma\in\operatorname{Gal}(\Q(\omega)/\Q)$, and as $\delta(\omega)\to\infty$ we have
    \[
    \left| \frac{1}{[\Q(\omega):\Q]}\sum_{\substack{\sigma\in\operatorname{Gal}(\Q(\omega)/\Q)\\ \omega^\sigma\in \operatorname{ct}^{-1}(\Delta)}}\operatorname{log}\big|P(\omega^\sigma)\big|-\int_\Delta\operatorname{log}\big|P(e(x))\big|\d x \right|\ll_{\Delta,P}\delta(\omega)^{-\kappa}.
    \]
\end{thm}

Moreover, an explicit value of $\kappa=\kappa(d,k)$ can be computed by applying \Cref{rmk: kappa} and \Cref{algorithmkappa}.

Other types of subsets of the full Galois orbits were considered in \cite{DimitrovHabegger}.
In \cite[Theorem 8.8]{DimitrovHabegger}, the case in which the subset is made by the orbit of $\omega$ under the action of a subgroup of the Galois group, was already studied. When the subset is restricted to a finite subgroup of $\mathbb{G}_m^d$, this was discussed in \cite[Theorem 1.2]{DimitrovHabegger}, which recovers \cite[Theorem 1.3]{lind2013homoclinic} by a different approach. It is worth to point out that \Cref{main thm} cannot be covered by the previous two cases.

This theorem was inspired by the work of R. Gualdi and M. Sombra \cite{RM}: they studied the distribution of the height of the intersection between the projective line defined by the linear polynomial $x+y+z$ and its translate by a torsion point. This serves as an example of how the height of the solution set is concluded from the arithmetic complexity of a system. They computed the limit of such heights of a sequence of torsion points by Bilu's equidistribution theorem \cite{bilu1997limit} and proposed a question regarding a quantitative version for their particular case \cite[Question 6.2]{RM}. It is shown in \cite[Section 5]{RM} that the formula of this height is locally of the form $\frac{1}{[\Q(\omega)/\Q]}\sum_\sigma\operatorname{log}|P(\omega^\sigma)|$, where $\sigma$ runs over the elements of the Galois group such that $\omega^{\sigma}$ corresponds to points in a triangle contained in unit square $[0,1)^2$. As an application of \Cref{main thm}, we solve \cite[Question 6.2]{RM} by deriving the following proposition.

\begin{prop}\label{main application}
    Let $(\omega_\ell)_{\ell\geq 1}$ be a strict sequence of non-trivial torsion points in $\mathbb{G}_m^2$, i.e. any proper algebraic subgroup of $\mathbb{G}_m^2$ contains $\omega_\ell$ for only finitely many values of $\ell$. For every $\omega_\ell=(\omega_{\ell,1},\omega_{\ell,2})$, let $\mathcal{P}(\omega_\ell)\in\mathbb{P}^2(\overline{\Q})$ be the solution of the system of linear equations
    \[
    x+y+z=x+\omega_{\ell,1}^{-1}y+\omega_{\ell,2}^{-1}z=0
    \]
    in the $2$-dimensional projective space. Then as $\ell\to\infty$ we have
    \[
    \left| \mathrm{h}(\mathcal{P}(\omega_\ell)) - \frac{2\zeta(3)}{3\zeta(2)} \right|\ll \delta(\omega_\ell)^{-1/(2^{61}\times 5^5)},
    \]
    where $\zeta$ is the Riemann zeta function.
\end{prop}

The following provides an overview of the proof of \Cref{main thm} and outlines the structure of this article. 

All notations are stated in \Cref{Notation Section}. Also, we have written down some formal definitions and preliminaries in this section.

A standard way from measure theory to estimate an equidistribution result is Koksma's inequality, see for example \cite[Theorem 5.4]{harman1998metric}. It says that if a function $f:[0,1)\to\R$ has bounded \textit{variation}, i.e.
$V(f)\coloneqq\operatorname{sup}_{0\leq a_0\leq\cdots\leq a_m< 1}\sum_{i=1}^m|f(a_i)-f(a_{i-1})|$
is bounded, then for every finite sequence $(x_1,x_2,\ldots,x_n)$ in $[0,1)$ we have
\begin{equation}\label{standard Koksma ineq}
    \left|\frac{1}{n}\sum_{i=1}^nf(x_i)-\int_{[0,1]}f(x) \d x\right|\leq D(x_1,\ldots,x_n)V(f),
\end{equation}
where 
\[
D(x_1,\ldots,x_n)\coloneqq\sup_{0\leq a<b\leq 1}\left| \frac{\sharp\{i:a\leq x_i< b\}
}{n}-(b-a) \right|
\]
is called the \textit{discrepancy} of $(x_1,\ldots,x_n)$. 
We say a sequence $(x_1,x_2,\ldots,x_n,\ldots)$ in $[0,1)$ is \textit{equidistributed} if $D(x_1,\ldots,x_n)\to 0$ as $n\to\infty$. Hence, if the quantity $D(x_1,\ldots,x_n)$ is smaller, then the finite sequence $(x_1,\ldots,x_n)$ is more ``equidistributed'' in $[0,1)$. 

However, in our case, we need to deal with functions having logarithmic singularities, which don't have bounded variation. It leads us to estimate the log function using the following ``bounded log function''
\[
\operatorname{log}_r(x)\coloneqq\operatorname{log}\operatorname{max}\{r,x\}\quad\text{for }r>0.
\]

The bounded log function also appears in \cite{DimitrovHabegger}, where some useful property, that we recall and enrich in \Cref{section proof}, are shown. By the triangle inequality, we have
\begin{multline}\label{eq: introduction_idea}
    \left| \frac{1}{n}\sum_{\substack{\sigma\in\operatorname{Gal}(\Q(\omega)/\Q)\\ \omega^\sigma\in \operatorname{ct}^{-1}(\Delta)}}\operatorname{log}|P(\omega^\sigma)|-\int_\Delta\operatorname{log}|P(e(x))|\d x \right|\\
    \leq \left| \frac{1}{n}\sum_{\substack{\sigma\in\operatorname{Gal}(\Q(\omega)/\Q)\\ \omega^\sigma\in \operatorname{ct}^{-1}(\Delta)}}\operatorname{log}_r|P(\omega^\sigma)|-\int_\Delta \operatorname{log}_r|P(e(x))|\d x \right| \\
    + \frac{1}{n}\left| \sum_{\substack{\sigma\in\operatorname{Gal}(\Q(\omega)/\Q)\\ \omega^\sigma\in \operatorname{ct}^{-1}(\Delta)}}\big( \operatorname{log}_r|P(\omega^\sigma)|-\operatorname{log}|P(\omega^\sigma)| \big) \right|\\
    + \left| 
\int_\Delta \big(\operatorname{log}_r|P(e(x))| -\operatorname{log}|P(e(x))|\big)\d x\right|.
\end{multline}

For the first difference in the right hand side of (\ref{eq: introduction_idea}), we need to develop a version of Koksma's inequailty for polytopes to estimate it, because the standard generalization of the Koksma's inequality (\ref{standard Koksma ineq}) only works for the whole $d$-dimensional interval $[0,1)^d$. There are some generalization of Koksma's inequality into more general sets like \cite{Koksma_derivatives} and \cite{harman2010variations}, but both of them have some shortcomings for our problems: \cite{Koksma_derivatives} contains the partial derivatives of functions that are difficult to compute in our case, and \cite{harman2010variations} somehow hides the geometric meaning of discrepancy. Therefore, based on the method in \cite[Proposition 7.1]{DimitrovHabegger}, in \Cref{Koksma section}, a version of Koksma's inequality over polytope is developed and applied to our bounded log function. As \cite[Proposition 7.1]{DimitrovHabegger} only works for continous function, we will construct ``continuous characteristic functions'' over polytopes for our purpose.

For the second difference in the right hand side of (\ref{eq: introduction_idea}), we use the triangle inequality again to divide it into two sums. The sum regarding $\operatorname{log}_r$ is easy to estimate because the function is bounded. For the sum involves log functions, the main theorem in \cite{DimitrovHabegger} will be applied. The estimation for the third difference is similar with the second one. 

In \Cref{section proof}, we put everything together to give a complete proof of \cref{main thm}.

\Cref{section application} contains a proof of \Cref{main application}, which gives an application of \Cref{main thm} and answers the question in the paper of R. Gualdi and M. Sombra we mentioned above \cite[Question 6.2]{RM}.

In \Cref{Appendix estimate log}, we present an algorithm that enables us to compute the explicit value of~ $\kappa$ in \Cref{main thm}.

~\\
\noindent\textbf{Acknowledgement}. 
This is part of the author's PhD project. She gratefully acknowledges many helpful suggestions and carefully reading provided by her supervisors, Roberto Gualdi and Walter Gubler. She also wishes to thank Riccardo Pengo for the reference suggestions regarding limits of Mahler measures.

\newpage

\section{Preliminaries}\label{Notation Section}

\subsection{Conventions and notations}

Throughout the article we use the following notations.

\begin{itemize}
    \item $|\cdot|$: maximum norm in $\R^d$ and $\operatorname{GL}_d(\Z)$. If $A=(a_{ij})\in\operatorname{GL}_d(\Z)$, then $|A|=\operatorname{max}_{i,j}|a_{i,j}|$.
    \item $|\cdot|_2$: Euclidean norm in $\R^d$.
    \item $\langle\cdot,\cdot\rangle$: Euclidean inner product on $\R^d$.
    \item A \textit{cubic ball} in $\R^d$ is a subset of the form $\{y\in\R^d:|x-y|\leq r\}$ for $x\in\R^d$ and $r\in\R$, where $x$ is called the center of the ball and $r$ is called the radius of the ball.
    \item $\operatorname{diam}(A)$: the diameter of a subset $A\subset\R^d$ with respect to maximum norm, which means $\operatorname{diam}(A)=\operatorname{max}\{|x-y|:x,y\in A\}$.
    \item $\mu$: Lebesgue measure.
    \item $\chi_A$: the characteristic function of a subset $A\subset\R^d$, defined as the real valued function given by $\chi_A(x)=1$ if $x\in A$ and $\chi_A(x)=0$ if $x\notin A$.
    \item $e(x)\coloneqq(e^{\mathrm{i}2\pi x_1},\ldots,e^{\mathrm{i}2\pi x_d})$ for $x=(x_1,\ldots,x_d)\in\R^d$.
    \item $m(P)$: the (logarithmic) Mahler measure for a Laurent polynomial $P$ in $d$ variables over~ $\C$ defined as 
    \[
    m(P)\coloneqq\int_{[0,1)^d}\operatorname{log}\big|P(e(x))\big|\d\mu(x).
    \]
    
    \item $|P|$: the maximal norm of the coefficient vector of a Laurent polynomial $P$.
    
    \item $\overline{\Q}$: a fixed algebraic closure of $\Q$.
    \item For $\omega=(\omega_1,\ldots,\omega_d)\in(\overline{\Q}^*)^d$, $a=(a_1,\ldots,a_d)\in\Z^d$ and $A=(a_{ij})\in\operatorname{GL}_d(\Z)$, we set
    \[
    \omega^a\coloneqq \omega_1^{a_1}\cdots\omega_d^{a_d}\quad\text{and}\quad \omega^A\coloneqq(\omega_1^{a_{11}}\cdots\omega_d^{a_{d1}},\ \ldots,\ \omega_1^{a_{1d}}\cdots\omega_d^{a_{dd}}).
    \]
    \item $\delta(\omega)$: the strictness degree of $\omega\in(\overline{\Q}^*)^d$ defined as
    \[
    \operatorname{inf}\left\{|a|:a\in\Z^d\backslash\{0\}\text{, }\omega^a=1\right\}.
    \] 
    
    \item $S^1\coloneqq\{z\in\C:|z|=1\}$.

    \item For $z\in S^1$ and $a=(a_1,\ldots,a_d)\in\Z^d$, we set $z^a=(z^{a_1},\ldots,z^{a_d})$.
    
    \item $\operatorname{log}_r(x)\coloneqq\operatorname{log}\operatorname{max}\{r,x\}$ for $r\in\R_{>0}$ and $x\in\R$, which is the bounded logarithm function.

    \item $\ll_{\ \bigcdot\ }$: the Vinogradov's notation where the constants implicit in it only depend on $\ \bigcdot\ $. For example, for a polynomial $P$, $\ll_P$ denotes the Vinogradov's notation, in which the implicit constants only depend on $P$.

    \item $\operatorname{ct}$: the cotropicalization map defined by
\[
(\C^*)^d\longrightarrow [0,1)^d, \quad (z_1,\ldots,z_d) \longmapsto (x_1,\ldots,x_d),
\]
where $z_i=r_ie^{\mathrm{i}2\pi x_i}$ for $r_i>0$ and $x_i\in[0,1)^d$. 

\end{itemize}

\subsection{Convex geometry}

Let $d\geq 1$ be an integer. In the statement of \cref{main thm}, we consider subsets of Galois orbits of torsion points that lie in the preimages of polytopes within $[0,1)^d$ under the cotropicalization map. Therefore, we need to introduce some basic definitions and properties from convex geometry.

Let $N\cong\Z^d$ be a lattice of rank $d$ and $M\coloneqq N^\vee$ be its dual. Then $N_\R\coloneqq N\otimes\R\cong\R^d$ is a real vector space of dimension $d$ and $M_\R\coloneqq M\otimes\R=N_\R^\vee$ is its dual space.

    A subset $C\subset N_\R$ is \textit{convex} if it contains all the segments with vertices in $C$. In other words, for any $u_1,u_2\in C$, we have the line segment
    \[
    \overline{u_1u_2}\coloneqq\{tu_1+(1-t)u_2:t\in[0,1]\}\subset C.
    \]

    An \textit{affine space} is of the form $\R v_1+\cdots+\R v_n+p$, where $\R v_1+\cdots+\R v_n$ is a linear space with $v_i\in\R^d$ and $p\in\R^d$. Let $C\subset\R^d$ be a convex subset. The \textit{affine hull} of $C$, denoted~ $\operatorname{aff}(C)$, is the minimal affine space which contains it. The \textit{dimension} of $C$ is defined as the dimension of its affine hull. The \textit{relative interior} of $C$, denoted $\operatorname{ri}(C)$, is defined as the interior of $C$ relative to its affine hull. 

    A convex subset $F\subset C$ is called a \textit{face} of $C$ if, for every closed line segment $\overline{u_1u_2}\subset C$ such that $\operatorname{ri}(\overline{u_1u_2})\cap F\neq\varnothing$, the inclusion $\overline{u_1u_2}\subset F$ holds. A face of $C$ of codimension $1$ is called a \textit{facet}. A face of $C$ of dimension $0$ is called a \textit{vertex}. A non-empty subset $F\subset C$ is called an \textit{exposed face} of $C$ if there exists $x\in M_\R$ such that
    \[
    F=\{u\in C:\langle x,u\rangle\leq\langle x,v\rangle, \forall v\in C\}.
    \]
    According to \cite[Theorem 1.5.8]{brondsted2012introduction}, any exposed face of $C$ is a face.

    Pick $x\in M_\R$ and $c\in\R$. Then the set
    \[
    H_{x,c}\coloneqq\{u\in N_\R:\langle x,u\rangle\geq c\}
    \]
    is called a \textit{closed half-space}. A \textit{polyhedron} in $N_\R$ is the intersection of finitely many closed half-spaces in $N_\R$. A \textit{polytope} in $N_\R$ is a bounded polyhedron in $N_\R$.  A polytope can be equivalently defined as the \textit{convex hull} of finitely many points. Indeed, a subset $P\subset N_\R$ is a polytope if and only if there exists $p_1,\ldots,p_s\in N_\R$ such that
\[
P=\operatorname{conv}(p_1,\ldots,p_s)\coloneqq\left\{ 
\sum_{i=1}^s\lambda_ip_i:\lambda_i\geq 0,\ \sum_{i=1}^s\lambda_i=1 \right\}.
\]
 We refer to \cite[Theorem 1.9.2]{brondsted2012introduction} for the equivalence. The structure of the faces of $P$ is given in \cite[Theorem 1.7.5]{brondsted2012introduction}. Specifically, every face of $P$ is of the form
 \[
 \{u\in P:\langle x,u\rangle=c \}
 \]
 for some $x\in M_\R$ and $c\in\R$.

\begin{defn}
    Let $P\subset N_\R$ be a polytope of the whole dimension $d$ and $F$ be a facet of $P$. Let $v_F$ be a vector orthogonal to $F$ with $|v_F|_2=1$. We set
    \[
    P_F\coloneqq\{u+\lambda v_F: u\in F,0\leq\lambda\leq 1\}.
    \]
    Then $P_F\subset N_\R$ is a polytope of dimension $d$. We call $\operatorname{vol}(F)\coloneqq\mu(P_F)$ the \textit{volume} of $F$ and the sum of the volumes of all facets of $P$ the \textit{surface area} of $P$.
\end{defn}

To prove our version of Koksma's inequality over polytopes, we will construct a continuous version of characteristic function of each polytope in \Cref{Koksma section}. The idea is to shrink the polytope by an arbitrary small scalar and connect the new smaller polytope and the original polytope by some linear functions. Therefore, we also introduce the notion of piecewise affine function.

    Recall that an \textit{affine function} $f:A\to B$ from an affine space $A$ to an affine space $B$ is a function which satisfies
    \[
    f(\lambda a+ (1-\lambda)c )=\lambda f(a)+(1- \lambda)f(c)
    \]
    for $a,c\in A$ and $\lambda \in\R$. In the more specific case where $A=L+x$ for a linear subspace $L$ of $\R^d$ and a point $x\in\R^d$, and $B=\R^d$, it follows from \cite[Section 1.1]{brondsted2012introduction} that there exists a linear function $g:L\to\R^d$ and a point $y\in\R^d$ such that $f(z+x)=g(z)+y$ for all $z\in L$.

    Let $P\subset N_\R$ be a polytope. A function $f:P\to\R$ is \textit{piecewise affine} if there is a finite cover of $P$ by closed subsets such that the restriction of $f$ to each of these closed subsets is the restriction of an affine function. Notice that every piecewise affine function is continuous. Hence, if we construct a piecewise affine characteristic function, then it will be continuous.

To show that our continuous version of characteristic function is piecewise affine, we first need to understand the set difference between the original polytope and the shrunk polytope. It will be described using a \textit{complete fan} consisting of \textit{polyhedral cones}.

A \textit{polyhedral cone} is a polyhedron $\sigma$ such that $\lambda\sigma=\sigma$ for all $\lambda\geq 0$. A \textit{fan} $\Sigma$ in $N_\R$ is a finite collection of polyhedral cones in $N_\R$ that such that
\begin{itemize}
    \item For all $\sigma\in\Sigma$, every face of $\sigma$ is also in $\Sigma$ and $\sigma$ does not contain any line.
    \item For all $\sigma,\sigma'\in\Sigma$, we have $\sigma\cap\sigma'$ is a face of both $\sigma$ and $\sigma'$.
\end{itemize}
If $\bigcup_{\sigma\in\Sigma}\sigma=N_\R$, then we say $\Sigma$ is \textit{complete}.

\subsection{Discrepancy}

Let $n\geq 1$ be an integer and let $x_1,\ldots,x_n$ be points in the hypercube $[0,1)^d$. We define the \textit{discrepancy} of the finite sequence $(x_1,\ldots,x_n)$ to be
\[
D(x_1,\ldots,x_n)\coloneqq\operatorname{sup}_{B}\left| \frac{\sum_{i=1}^n\chi_B(x_i)}{n}-\mu(B) \right|,
\]
where $B$ ranges over all products $\prod_{i=1}^d[a_i,b_i)$ with $0\leq a_i< b_i\leq 1$. By definition, the discrepancy lies in $[0,1]$. In some references, however, for example in \cite{harman1998metric}, the discrepancy is defined without normalization by $n$, thus allowing values greater than $1$.

In this paper, we consider the discrepancy of the arguments of the Galois orbit of a torsion point $\omega\in\mathbb{G}_m^d$. It has been estimated in \cite[Proposition 3.3]{DimitrovHabegger}:

\begin{prop}\label{discrepancy over Galois orbit}
    Let $\omega\in\mathbb{G}_m^d$ be a torsion point of order $n$, and let $\{\omega^\sigma:\sigma\in\operatorname{Gal}(\Q(\omega)/\Q)\}=\{e(x_i):1\leq i\leq n\}$, where all $x_i\in[0,1)^d$. Then we have
    \[
    D(x_1,\ldots,x_n)\ll_d \frac{(\operatorname{log}2\delta(\omega))^{d-1}\operatorname{log}\operatorname{log}3\delta(\omega)}{\delta(\omega)^{1/2}}
    \]
    as $\delta(\omega)\to\infty$.
\end{prop}

As we are concerned about polytopes, the following notion of isotropic discrepancy will also be useful. The reference for it is \cite[Chapter 2]{kuipers2012uniform}.

\begin{defn}
    Let $d,n\in\N$ and $x_1,\ldots,x_n\in[0,1)^d$. The \textit{isotropic discrepancy} of the finite sequence $(x_1,\ldots,x_n)$ is
    \[
    J(x_1,\ldots,x_n)\coloneqq\operatorname{sup}_C\left|\frac{ \sum_{i=1}^n \chi_C(x_i) }{n}-\mu(C)\right|
    \]
    where $C$ ranges over all convex subsets of $[0,1)^d$.
\end{defn}

\begin{prop}\label{isotropical discrepancy and discrepancy}
    Let $d,n\in\N$ and $x_1,\ldots,x_n\in[0,1)^d$. Then we have
    \[
    D(x_1,\ldots,x_d)\leq J(x_1,\ldots,x_n)\leq (4d\sqrt{d}+1)D(x_1,\ldots,x_n)^{1/d}.
    \]
\end{prop}
\begin{proof}
    The first inequality $D(x_1,\ldots,x_d)\leq J(x_1,\ldots,x_n)$ follows immediately from the definitions. For the second inequality, we refer the reader to \cite[Theorem 2.1.6]{kuipers2012uniform}.
\end{proof}

\subsection{Essentially atoral polynomials}

The main theorem \cref{main thm} in this article is a quantitative version of equidistribution of function $\operatorname{log}|P(e(x))|$ where $P$ is an \textit{essentially atoral Laurent polynomial}. We give the definition of essentially atoral Laurent polynomial here.

\begin{defn}\label{essentially atoral}
    A \textit{torsion coset} of $\mathbb{G}_m^d$ is a coset $\omega H$, where $\omega\in\mathbb{G}_m^d$ is a torsion point and $H$ is a connected algebraic subgroup of $\mathbb{G}_m^d$. A torsion coset is \textit{proper} if it is not equal to $\mathbb{G}_m^d$. A non-zero Laurent polynomial $P\in\C[T_1^{\pm1},\ldots,T_d^{\pm1}]$ is called \textit{essentially atoral} if the Zariski closure of
    \[
    \left\{(z_1,\ldots,z_d)\in\C^d:P(z_1,\ldots,z_d)=0\right\}\cap (S^1)^d
    \]
    in $\mathbb{G}_m^d$ is a finite union of proper torsion cosets and irreducible algebraic sets of codimension at least $2$.
\end{defn}

\begin{ex}
    If $d=1$, then a non-zero Laurent polynomial $P\in\C[T^{\pm 1}]$ is essentially atoral if and only if it does not vanish at any non-torsion point in $S^1$.
\end{ex}

\begin{ex}\label{P i are essentially atoral}
    Let $d\geq2$ and $P(T_1,\ldots,T_d)=T_1^{m_1}\cdots T_d^{m_d}-T_1^{n_1}\cdots T_d^{n_d}$ be a nonzero Laurent binomial over $\C$, where all $m_i,n_i$ are integers. Then $P$ is essentially atoral. Indeed, let
    \[
    a\coloneqq(m_1-n_1,\ldots,m_d-n_d)\quad\text{and}\quad H_a\coloneqq\{(z_1,\ldots,z_d)\in(\C^*)^d:z_1^{a_1}\cdots z_d^{a_d}=1\}.
    \]
    By \cite[Proposition 3.2.10]{Bombieri_Gubler_2006}, there exists a finite subset $S\subset S^1$ such that $H_a\cong S\times\mathbb{G}_m^{d-1}$. Therefore, the Zariski closure of
    \[
    \left\{(z_1,\ldots,z_d)\in\C^d:P(z_1,\ldots,z_d)=0\right\}\cap (S^1)^d= H_a\cap(S^1)^d\cong S\times (S^1)^{d-1}
    \]
    is isomorphic to a finite union of $\mathbb{G}_m^{d-1}$ and each $\mathbb{G}_m^{d-1}$ is a proper torsion coset of $\mathbb{G}_m^d$.
\end{ex}

\subsection{Height}

Height functions assign a notion of arithmetic size of complexity to mathematical objects, such as points and cycles in varieties. In this paper, we are interested in Weil height of points in projective space.

Let $K$ be a number field. The notion of absolute values of $K$ will be used in the definition of height functions. A standard reference for absolute values is \cite[Chapter 2]{neukirch2013algebraic}. Here, we normalize the absolute values of $K$ as follows. Let $v$ be a non-archimedean place of $K$. Then there exists a prime number $p$ such that $v$ extends the $p$-adic place. Recall that $|p|_p=\frac{1}{p}$. For any $x\in K$, we define
\[
|x|_v\coloneqq\big| N_{K_v/\Q_p}(x) \big|_p^{1/[K:\Q]}
\]
where $K_v$ (resp. $\Q_p$) is the completion of $K$ with respect to $v$ (resp. $\Q$ with respect to the $p$-adic place). 

Let $M_K$ be the set of all places on $K$. For example, if $K=\Q$, then $M_{\Q}$ consists of an archimedean place $\infty$ and non-archimedean places $v_p$ corresponding to prime numbers $p$.

\begin{defn}
    For a point $x=[x_0:\cdots:x_d]\in\mathbb{P}^d(\overline{\Q})$ in $d$-dimensional projective space over $\overline{\Q}$, let $K$ be a number field containing all coordinates of $x$. Then the \textit{height} of $x$ is defined as
    \[
    \mathrm{h}(x)\coloneqq \sum_{w\in M_K} \operatorname{max}_j \operatorname{log}|x_j|_w.
    \]
    For each place $w\in M_K$, the $w$-\textit{adic height} of the homogeneous coordinates vector $(x_0,\ldots,x_d)$ of $x$ is defined as
    \[
    \mathrm{h}_w(x_0,\ldots,x_d)\coloneqq \operatorname{max}_j \operatorname{log}|x_j|_w.
    \]
    For each place $v\in M_\Q$, the $v$-\textit{adic height} of the homogeneous coordinates vector $(x_0,\ldots,x_d)$ of $x$ is defined as
    \[
    \mathrm{h}_v(x_0,\ldots,x_d)=\sum_{w\text{ extends }v}\mathrm{h}_w(x_0,\ldots,x_d).
    \]
\end{defn}

The height is well-defined in this way, see \cite[Chapter 1]{Bombieri_Gubler_2006} for details.

\newpage

\section{Koksma's inequality over polytopes}\label{Koksma section}

\pdfsuppresswarningpagegroup=1

Throughout this section, we identify points and vectors in the real vector space $\R^d$.

\begin{defn}
    Let $f:B\to\R$ be a function, where $B$ is a non-empty subset of $\R^d$. Then the \textit{modulus of continuity} of $f$ is defined by
    \[
    \rho(f,t)=\operatorname{sup}\big\{|f(x)-f(y)|:x,y\in B,\ |x-y|\leq t\big\},\quad\forall t\geq 0.
    \]
\end{defn}

Let $f:[0,1]^d\to\R$ be a continuous function, and let $(x_1,\ldots,x_n)$ be a finite sequence in $[0,1)^d$ with discrepancy $D=D(x_1,\ldots,x_n)$. Then \cite[Proposition 7.1]{DimitrovHabegger} gives the following proposition, which shares the same idea with simplest Koksma's inequality in dimension one (\ref{standard Koksma ineq}), i.e. the upper bound of 
\[
\left|\frac{1}{n}\sum_{i=1}^nf(x_i)-\int_{[0,1)^d} f\d\mu\right|
\]
is related to the discrepancy $D$ and the variation of $f$.

\begin{prop}\label{a Koksma in hypercube}
    Let $f:[0,1]^d\to\R^d$ be a continuous function and let $(x_1,\ldots,x_n)$ be a finite sequence in $[0,1)^d$ with discrepancy $D=D(x_1,\ldots,x_n)$. Then
    \[
    \left|\frac{1}{n}\sum_{i=1}^nf(x_i)-\int_{[0,1)^d} f\d\mu\right|\leq (1+2^{d+1})\rho\left(f,D^{\frac{1}{d+1}}\right).
    \]
\end{prop}

Now let $\Delta\subset [0,1]^d$ be a polytope of full dimension, i.e. $\operatorname{dim}(\Delta)=d$. In this section, we would like to find an upper bound for the difference
\[
\frac{1}{n}\sum_{\substack{i=1 \\ x_i\in\Delta}}^nf(x_i)-\int_\Delta f\d\mu=\frac{1}{n}\sum_{i=1}^n f(x_i)\chi_\Delta(x_i)-\int_{[0,1)^d}f\chi_\Delta\d\mu,
\]
where $\chi_\Delta$ is the characteristic function of $\Delta$ such that $\chi_\Delta(x)=1$ if $x\in\Delta$ and $\chi_\Delta(x)=0$ otherwise.

In order to apply \Cref{a Koksma in hypercube} to our case, we need to construct a continuous version of characteristic function and compute its modulus of continuity. The proof below works for the case $d\geq 2$. However, the result we proved also works for $d=1$ and the proof is essentially the same. The only difference is that if $d=1$ then we don't need to intersect the polytope with $2$-dimensional planes to reduce some problems to the $2$-dimensional case. 

\subsection{Construction of ``continuous characteristic function"}\label{subsection: Construction of continuous characteristic function}

Consider $0<\epsilon<1$. Let $\Delta\subset\R^d$ be a polytope. Let us construct a continuous function $\chi_{\Delta,\epsilon}^{c}$ to estimate the characteristic function $\chi_\Delta$. The main idea is to shrink the polytope $\Delta$ to $\Delta_\epsilon$ in a linear way. The function $\chi_{\Delta,\epsilon}^c$ will be $0$ outside $\Delta$ and $1$ inside $\Delta_\epsilon$. We will then connect the facets of $(\Delta,0)$ and $(\Delta_\epsilon,1)$ by linear functions. In order to define the function, we need to pick a special point in the interior of $\Delta$. Recall that a cubic ball in $\mathbb{R}^d$ is of the form~ $\{ y\in\R^d:|x-y|\leq r \}$ for $x\in\mathbb{R}^d$ and $r\in\R_{>0}$ using the maximum norm $|\cdot|$.

\begin{defn}
    An \textit{inscribed cubic ball} of a polytope is a ball of maximal radius that is contained within the polytope.
\end{defn}

\begin{prop}
    Every polytope has an inscribed cubic ball.
\end{prop}
\begin{proof}
    Let $P\subset\R^d$ be a polytope. If $P$ is a point, then the only point in the inscribed cubic ball is $P$. We suppose that $P$ is not a point. First, we claim that for every point $x\in\operatorname{ri}(P)$ there exists a ball centered at $x$ of maximal radius that is contained within $P$. Indeed, for every $x\in\operatorname{ri}(P)$, there exists $r_x\geq0$ such that
    \[
    r_x=\operatorname{min}\left\{|x-y|:y\in P\backslash\operatorname{ri}(P)\right\},
    \]
    since 
    \[
    P\backslash\operatorname{ri}(P)\to\R,\quad y\mapsto|x-y|
    \]
    is a continuous function on the compact set $P\backslash\operatorname{ri}(P)$. Then the ball centered at $x$ with radius~ $r_x$ is contained within $P$, and it has maximal radius among all balls centered at $x$ that are contained within $P$. The claim follows. 

    Consider the function
\[     r:P\to\R,\quad x\mapsto \operatorname{min}\left\{  |x-y|:y\in P\backslash\operatorname{ri}(P) \right\}.     \] 
    Note that $r(x)=r_x$ if $x\in\operatorname{ri}(P)$ and $r(x)=0$ if $x\in P\backslash\operatorname{ri}(P)$. Let us show that $r$ is continuous. Take $x_0,x\in P$. For any $y\in P\backslash\operatorname{ri}(P)$, we have
    \[
    |x-y|\leq |x-x_0|+|x_0-y|
    \]
    and thus
    \[
    \operatorname{min}\left\{ |x-y|:y\in P\backslash\operatorname{ri}(P) \right\}\leq\operatorname{min}\left\{ |x_0-y|:y\in P\backslash\operatorname{ri}(P) \right\}+|x-x_0|.
    \]
    Hence $r(x)-r(x_0)\leq|x-x_0|$. Symmetrically, we have $r(x_0)-r(x)\leq |x-x_0|$ and therefore $r$ is continuous. The proposition follows from the fact that $P$ is compact.
\end{proof}

\begin{defn}\label{defn: inner radius}
    Let $P\subset\R^d$ be a polytope. Then the \textit{inner radius} of $P$, denoted $\operatorname{inrad}(P)$, is the radius of an inscribed cubic ball of $P$.
\end{defn}

Let $x_c$ be the center of an inscribed cubic ball of $\Delta$. For each $0<\epsilon<1$, we define
\begin{align*}
    \varphi_\epsilon:\R&\to\R^d\\
    x&\mapsto x_c+(1-\epsilon)(x-x_c)=(1-\epsilon)x+\epsilon x_c,
\end{align*}
which is the shrinking of $\R^d$ by a scalar $1-\epsilon$ with a specified center at $x_c$. Set $\Delta_\epsilon\coloneqq \varphi_\epsilon(\Delta)$.

\begin{lem}
    For each $0<\epsilon<1$, the set $\Delta_\epsilon$ is a polytope contained in $\Delta$.
\end{lem}
\begin{proof}
    It is stated in \cite[Exercise 2.8.1]{brondsted2012introduction} that the image of a polytope under an affine function remains a polytope. Since $\varphi_\epsilon$ is an affine function, the set $\Delta_\epsilon=\varphi_\epsilon(\Delta)$ is a polytope. By the convexity of $\Delta$ and the fact that $x_c\in\Delta$, we have $\Delta_\epsilon\subset\Delta$.
\end{proof}

\begin{figure}[H]
    \centering
    \columnwidth=6cm
    \def\svgwidth{\columnwidth}
    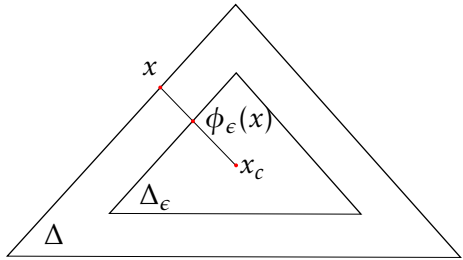

    \caption{An Example of $\Delta_\epsilon\subset\Delta$ in dimension $2$}
    \label{Example of Delta epsilon}
\end{figure}

\begin{rmk}
    One can prove that $\Delta_\epsilon\to\Delta$ as $\epsilon\to 0$ in terms of \textit{Hausdorff distance}. See \cite[Chapter 4]{rockafellar2009variational} for more details on Hausdorff distance. Indeed, the Hausdorff distance between $\Delta$ and $\Delta_\epsilon$ is
    \[
    d_H(\Delta_\epsilon,\Delta)\coloneqq\operatorname{max}\left\{ \operatorname{sup}_{x\in \Delta}d(x,\Delta_\epsilon),\ \operatorname{sup}_{y\in \Delta_\epsilon}d(\Delta,y) \right\},
    \]
    where
    \begin{align*}
        &d(\Delta,y)\coloneqq\operatorname{inf}_{x\in\Delta}|x-y|=0,\quad\forall y\in\Delta_\epsilon,\\
        &d(x,\Delta_\epsilon)\coloneqq\operatorname{inf}_{y\in\Delta_\epsilon}|x-y|\leq|x-\varphi_\epsilon(x)|=\epsilon|x-x_c|,\quad\forall x\in\Delta.
    \end{align*}
    Therefore, we have
    \[
    d_H(\Delta_\epsilon,\Delta)\leq\epsilon\cdot\operatorname{sup}_{x\in\Delta}|x-x_c|\leq\epsilon\cdot\operatorname{diam}(\Delta).
    \]
    It follows that $d_H(\Delta,\Delta_\epsilon)\to 0$ as $\epsilon\to 0$.
\end{rmk}

Before constructing the continuous version of characteristic function, let us first study the set difference $\Delta\backslash\operatorname{ri}(\Delta_\epsilon)$. It will help us to show that the constructed function is piecewise affine and thus continuous.

\begin{defn}[truncated hyperpyramid]
    Let $B\subset\R^d$ be a $(d-1)$-dimensional polytope that is the convex hull of the points $p_1,\ldots,p_s\in\R^d$, and let $x\in\R^d\backslash B$. The polytope
    \[
    P\coloneqq\operatorname{conv}(p_1,\ldots,p_s,x)
    \]
    is called a \textit{hyperpyramid} with \textit{base} $B$ and \textit{peak} $x$. Let $H\subset\R^d$ be a hyperplane such that $H$ is parallel to $B$ and $H\cap P\notin\big\{ \varnothing,\{x\},B \big\}$. Then there exist $q_1,\ldots,q_s\in\R^d$ such that $H\cap P=\operatorname{conv}(q_1,\ldots,q_s)$. The polytope
    \[
    \operatorname{conv}(p_1,\ldots,p_s,q_1,\ldots,q_s)
    \]
    is called a \textit{truncated hyperpyramid} with \textit{top} $H\cap P$, \textit{base} $B$ and \textit{peak} $x$.
\end{defn}

\begin{lem}\label{parallel lem: F and phi epsilon F}
    Let $F$ be a facet of $\Delta$. Then $\varphi_\epsilon(F)$ and $F$ are parallel for every $0<\epsilon<1$.
\end{lem}
\begin{proof}
    Let $v\in\R^d$ such that $v$ is orthogonal to $F$. Then for any $x,y\in F$, we have $\langle v,x-y\rangle=0$. Thus $\langle v,\varphi_\epsilon(x)-\varphi_\epsilon(y)\rangle=(1-\epsilon)\langle v,x-y\rangle=0$, which means that $v$ is orthogonal to $\varphi_\epsilon(F)$. It follows that $F$ and $\varphi_\epsilon(F)$ are parallel.
\end{proof}

\begin{lem}\label{mid: union of polytopes}
    Let $F_1,\ldots,F_e$ be all the facets of $\Delta$. For each $0<\epsilon<1$, the set $\Delta\backslash\operatorname{ri}(\Delta_\epsilon)$ can be written as a union of truncated hyperpyramids $P_1,\ldots,P_e$, where each $P_i$ has top $\varphi_\epsilon(F_i)$, base $F_i$ and peak $x_c$.
\end{lem}
\begin{proof}
    For each $1\leq i\leq e$, let us consider the polyhedral cone
    \[
    \angle(F_i)\coloneqq\left\{ t(x-x_c):x\in F_i,\ t\geq 0 \right\}.
    \]
    Let $\Sigma$ be the collection of $\angle(F_1),\ \ldots,\ \angle(F_e)$ and all their faces. Then $\Sigma$ is a complete fan. By the completeness of $\Sigma$, we have
    \[
    \Delta\backslash\operatorname{ri}(\Delta_\epsilon)=\big( \angle(F_1)\cap(\Delta\backslash\operatorname{ri}(\Delta_\epsilon)) \big) \cup\cdots\cup \big( \angle(F_e)\cap(\Delta\backslash\operatorname{ri}(\Delta_\epsilon)) \big).
    \]
    It follows from \Cref{parallel lem: F and phi epsilon F} that each $\angle(F_i)\cap(\Delta\backslash\operatorname{ri}(\Delta_\epsilon))$ is a truncated hyperpyramid with top~ $\phi_\epsilon(F_i)$, base $F_i$ and peak $x_c$.
\end{proof}

The following lemma will be used when we prove our version of Koksma's inequality over polytopes.

\begin{lem}\label{volume bet Delta and Delta_epsilon}
    Let $0<\epsilon<1$. Let $S(\Delta)$ be the surface area of $\Delta\subset\R^d$. Then
    \[
    \mu\left(\Delta\backslash\Delta_\epsilon\right)\leq\frac{\epsilon S(\Delta)\operatorname{diam}(\Delta)}{\sqrt{d}}.
    \]
\end{lem}
\begin{proof}
    By \Cref{mid: union of polytopes}, $\Delta\backslash\operatorname{ri}(\Delta_\epsilon)$ can be written as the union of truncated hyperpyramid $P_i$, where each $P_i$ has top $\varphi_\epsilon(F_i)$, base $F_i$ and peak $x_c$. For each $i$, we can pick a point $x\in \operatorname{aff}(F_i)$ such that the vector $x_c-x$ is orthogonal to $F_i$. 
    
    The volume of a hyperpyramid in $\R^d$ with base $B$ and peak $p$ can be computed using \cite[Theorem 1.2.10]{truncatedhyperpuramid}, which says that such a hyperpyramid has volume $\operatorname{vol}(B)\cdot h/d$, where $h$ is the Euclidean distance from $p$ to $B$. Applying this theorem, we get
    \[
    \mu(P_i)\leq\frac{\operatorname{vol}(\phi_\epsilon(F_i))|\varphi_\epsilon(x)-x|_2}{d}=\frac{\operatorname{vol}(\phi_\epsilon(F_i))\epsilon|x_c-x|_2}{d}\leq\frac{\operatorname{vol}(F_i)\epsilon\cdot\operatorname{sup}\{|a-b|_2:a\in\Delta,b\in\Delta\}}{d}.
    \]
    Recall that $\operatorname{diam}(\Delta)$ is the diameter of $\Delta$ with respect to the maximum norm. For every $a,b\in\Delta$, we have $|a-b|_2\leq\sqrt{d}|a-b|$. It implies that $\operatorname{sup}\{|a-b|_2:a\in\Delta,b\in\Delta\}\leq\sqrt{d}\operatorname{diam}(\Delta)$ and thus
    \[
    \mu(P_i)\leq\frac{\operatorname{vol}(F_i)\epsilon\operatorname{diam}(\Delta)}{\sqrt{d}},
    \]
    which proves the result.
\end{proof}

We now construct the continuous version of characteristic function of $\Delta$ with respect to~ $\epsilon$. For every $y\in\Delta\backslash\Delta_\epsilon$, consider the ray
$\left\{x_c+\lambda(y-x_c):\lambda\in\R_{\geq 0}\right\}$ with starting point $x_c$ and passing through $y$. Take
\[
\lambda_y\coloneqq\operatorname{sup}\left\{\lambda\in\R_{\geq 0}:x_c+\lambda(y-x_c)\in\Delta\right\}.
\]
and set $x_y\coloneqq x_c+\lambda_y(y-x_c)\in\Delta$. Then we have
\[
y=x_c+(1-\delta_y)(x_y-x_c)=\varphi_\delta(x_y)
\]
for $\delta_y\coloneqq 1-\frac{1}{\lambda_y}$. Hence $y\in \varphi_\delta(\Delta)=\Delta_\delta$. As $y\in\Delta\backslash\Delta_\epsilon$, we have $0\leq\delta_y<\epsilon$. With this expression, we can define a function $\chi_{\Delta,\epsilon}^c:\R^d\to\R$ such that

\begin{align}\label{construction of cts char fun}
    \chi_{\Delta,\epsilon}^c(y)\coloneqq\left\{
    \begin{array}{cc}
        0 & \text{for }y\notin\Delta, \\
        \frac{\delta_y}{\epsilon} & \text{for }y\in\Delta\backslash\Delta_\epsilon,\\
        1 & \text{for }y\in\Delta_\epsilon.
    \end{array}
    \right.
\end{align}

\begin{prop}\label{piecewise affine}
    For each $0<\epsilon<1$, the function $\chi_{\Delta,\epsilon}^c:\R^d\to\R$ is piecewise affine. 
\end{prop}
\begin{proof}
    By definition, $\chi_{\Delta,\epsilon}^c=0$ on $\R^d\backslash\Delta$ and $\chi_{\Delta,\epsilon}^c=1$ on $\Delta_\epsilon$. By \Cref{mid: union of polytopes}, we have
    \[
    \Delta=\Delta_\epsilon\cup(P_1\cup\cdots\cup P_e),
    \]
    where every $P_i$ is a truncated hyperpyramid. Let us consider $\chi_{\Delta,\epsilon}^c$ on each closed subset $P_i$.

    Suppose $y_1,y_2\in P_i$ with $\chi_{\Delta,\epsilon}^c(y_1)=\frac{\delta_1}{\epsilon}$ and $\chi_{\Delta,\epsilon}^c(y_2)=\frac{\delta_2}{\epsilon}$. Let $y=ty_1+(1-t)y_2\in P_i$ for $t\in\R$. Then there exists $0\leq\delta<\epsilon$ such that $\chi_{\Delta,\epsilon}^c(y)=\frac{\delta}{\epsilon}$. It is easy to check that $\delta=t\delta_1+(1-t)\delta_2$ using elementary Euclidean geometry. Therefore,
    \begin{align*}
        \chi_{\Delta,\epsilon}^c(y)&= \frac{\delta}{\epsilon}\\
        &= t\frac{\delta_1}{\epsilon} + (1-t)\frac{\delta_2}{\epsilon} \\
        &=t\chi_{\Delta,\epsilon}^c(y_1) + (1-t)\chi_{\Delta,\epsilon}^c(y_2)
    \end{align*}
    and so $\chi_{\Delta,\epsilon}^c$ is affine on $P_i$. The conclusion is proved.
\end{proof}

\begin{cor}
    For each $0<\epsilon<1$, the function $\chi_{\Delta,\epsilon}^c$ is continuous.
\end{cor}

\begin{rmk}
    One can prove that $\chi_{\Delta,\epsilon}^c$ converges pointwise to the characteristic function $\chi_\Delta$ in the interior of $\Delta$ as $\epsilon\to 0$. Indeed, take $y\in\operatorname{ri}(\Delta)$. Then there exists $\epsilon>0$ and $x\in\Delta\backslash\operatorname{ri}(\Delta)$ such that
    \[
    y=x_c+(1-\epsilon)(x-x_c),
    \]
    where $x$ is in the intersection of $\Delta\backslash\operatorname{ri}(\Delta)$ and the ray with starting point $x_c$ and passing through $y$. Then $y\in\Delta_\epsilon$. Therefore, for any $0<\delta<\epsilon$, we have $y\in\Delta_\epsilon\subset\Delta_\delta$ and therefore~ $\chi_{\Delta,\delta}^c(y)=1=\chi_\Delta(y)$, which proves the desired result.
\end{rmk}

\subsection{Computation of modulus of continuity}\label{Subsection: Computation of modulus of continuity}

To apply \Cref{a Koksma in hypercube}, we need to compute the modulus of continuity $\rho(f\cdot\chi_{\Delta,\epsilon}^c,t)$ for~ $\epsilon>0$ and $t\geq 0$. First, let us compute $\rho(\chi_{\Delta,\epsilon}^c,t)$. As we will apply our version of Koksma's inequality over polytopes to a sequence in $[0,1)^d$ with small enough discrepancy, it is sufficient to compute~ $\rho(\chi_{\Delta,\epsilon}^c,t)$ when $t$ is small enough. 

\begin{prop}\label{prop:estimate of modulus of continuity}
    Let $\Delta\subset\R^d$ be a polytope, and let $\chi_{\Delta,\epsilon}^c$ be the continuous version of the characteristic function of $\Delta$ that we constructed in (\ref{construction of cts char fun}) for $0<\epsilon<1$. Then for sufficiently small real numbers $\epsilon>0$ and $t\geq 0$, we have
    \[
    \rho(\chi_{\Delta,\epsilon}^c,t)\leq\frac{t}{\epsilon\cdot\operatorname{inrad}(\Delta)}.
    \]
\end{prop}

Let us now prove \Cref{prop:estimate of modulus of continuity}. During the proof, we intersect the polytopes with affine two-dimensional planes multiple times to make our arguments visual. 

\begin{defn}
    A subset $H\subset \R^d$ is an \textit{affine two-dimensional plane} if there exist $x,y,z\in H$ such that $x,y,z$ are not on the same line and
    \[
    H=\left\{ x+\lambda_1(x-y)+\lambda_2(y-z):\lambda_1\in\R,\ \lambda_2\in\R \right\}.
    \]
\end{defn}

\begin{lem}\label{intersecting 2-dim plane --> polytope}
    Let $H\subset \R^d$ be an affine $2$-dimensional plane. Let $P\subset \R^d$ be a $d$-dimensional polytope and let $F$ be a face of $P$. Then $P\cap H$ is a polytope and $F\cap H$ is a face of $P\cap H$. 
\end{lem}
\begin{proof}
    Notice that $H$ is a polyhedron. Since $P$ is a bounded polyhedron, then $P\cap H$ is a bounded polyhedron and thus a polytope. Let us show that $F\cap H$ is a face of $P\cap H$. Let~ $\overline{u_1u_2}\subset P\cap H$ be a closed segment such that $\operatorname{ri}(\overline{u_1u_2})\cap(F\cap H)\neq\varnothing$. Then $\overline{u_1u_2}\subset P$ is a closed segment such that $\operatorname{ri}(\overline{u_1u_2})\cap F\neq\varnothing$. Since $F$ is a face of $P$, then $\overline{u_1u_2}\subset F$. Therefore,~ $\overline{u_1u_2}\subset F\cap H$, which means $F\cap H$ is a face of $P\cap H$.
\end{proof}

By definition,
\[
\rho(\chi_{\Delta,\epsilon}^c,t)=\operatorname{sup}\left\{ \left|\chi_{\Delta,\epsilon}^c(x)-\chi_{\Delta,\epsilon}^c(y)\right|:x,y\in\R^d,\ |x-y|\leq t \right\}.
\]
Let $x,y\in\R^d$ with $|x-y|\leq t$. Our goal is to find an upper bound for $|\chi_{\Delta,\epsilon}^c(x)-\chi_{\Delta,\epsilon}^c(y)|$. If~ $x,y\notin\Delta$, then $|\chi_{\Delta,\epsilon}^c(x)-\chi_{\Delta,\epsilon}^c(y)|=|0-0|=0$. If $x,y\in\Delta_\epsilon$, then $|\chi_{\Delta,\epsilon}^c(x)-\chi_{\Delta,\epsilon}^c(y)|=|1-1|=0$. 

If $x\notin\Delta$ and $y\in\Delta\backslash\operatorname{ri}(\Delta_\epsilon)$, then by considering the line passing through $x$ and $y$, we can find an $x'\in\Delta\backslash\operatorname{ri}(\Delta)$ such that $|x'-y|\leq |x-y|\leq t$ and $\chi_{\Delta,\epsilon}^c(x)=\chi_{\Delta,\epsilon}^c(x')=0$. The case $y\notin\Delta$ and $x\in\Delta\backslash\operatorname{ri}(\Delta_\epsilon)$ is in the same way.

Similarly, if $x\in\Delta\backslash\operatorname{ri}(\Delta_\epsilon)$ and $y\in\Delta_\epsilon$, then by considering the line passing through $x$ and~ $y$, we can find an $y'\in\Delta_\epsilon\backslash\operatorname{ri}(\Delta_\epsilon)$ such that $|x-y'|\leq|x-y|\leq t$ and $\chi^c_{\Delta,\epsilon}(y)=\chi^c_{\Delta,\epsilon}(y')=1$. The case $y\in\Delta\backslash\operatorname{ri}(\Delta_\epsilon)$ and $x\in\Delta_\epsilon$ is in the same way.

Hence, it remains to consider the case $x,y\in\Delta\backslash\operatorname{ri}(\Delta_\epsilon)$.

Recall that $x_c$ is the center of an inscribed cubic ball of $\Delta$. 

\ 

\noindent\textbf{Reduction Step}. Let $F_1,\ldots,F_e$ be the facets of $\Delta$. By \Cref{mid: union of polytopes}, the set $\Delta\backslash\operatorname{ri}(\Delta_\epsilon)$ is a union of truncated hyperpyramids $P_1,\ldots,P_e$, where each $P_i$ has top $\varphi_\epsilon(F_i)$, base $F_i$ and peak $x_c$. We may assume that $x,y\in P_i$ for one $i$. 

Indeed, otherwise $x,y$ are not in the same $P_i$ for one $i$. When $t$ is small enough, $x,y,x_c$ are not on the same line. Let $H$ be the affine two-dimensional plane containing $x,y,x_c$. By definition $\chi_{\Delta,\epsilon}^c(x)=\frac{\delta_x}{\epsilon}$ and $\chi_{\Delta,\epsilon}^c(y)=\frac{\delta_y}{\epsilon}$ for $0\leq\delta_x,\delta_y<\epsilon$. If $\delta_x=\delta_y$, then $|\chi_{\Delta,\epsilon}^c(x)-\chi_{\Delta,\epsilon}^c(y)|=0$. Without loss of generality, we suppose that $\delta_x<\delta_y$.

We claim that $\varphi_{\delta_x}(\Delta)\cap H$ is a $2$-dimensional polytope with facets $\varphi_{\delta_x}(F_{k_1})\cap H,\ldots,\varphi_{\delta_x}(F_{k_l})\cap H$, where $F_{k_1},\ldots,F_{k_l}$ are the facets of $\Delta$ such that $\varphi_{\delta_x}(F_{k_i})\cap H$ is a line segment. Indeed, by \Cref{intersecting 2-dim plane --> polytope}, each nonempty $\varphi_{\delta_x}(F_j)\cap H$ is a face of the polytope $\varphi_{\delta_x}(\Delta)\cap H$. 
Note that~ $\overline{xy}\subset\varphi_{\delta_x}(\Delta)\cap H$ and $\overline{xx_c}\subset\varphi_{\delta_x}(\Delta)\cap H$ are two line segments which do not lie on the same line. Therefore, we have $\operatorname{dim}(\varphi_{\delta_x}(\Delta)\cap H)=2$. Since $\operatorname{dim}(\varphi_{\delta_x}(F_{k_i})\cap H)=1$, then~ $\varphi_{\delta_x}(F_{k_i})\cap H$ is a facet of $\varphi_\delta(\Delta)\cap H$.

Similarly, $\varphi_{\delta_y}(\Delta)\cap H$ is a $2$-dimensional polytope with facets $\varphi_{\delta_y}(F_{k_1})\cap H,\ldots,\varphi_{\delta_y}(F_{k_l})\cap H$. Let
\[
d_i\coloneqq\operatorname{inf}\big\{ |a-b|:a\in\operatorname{aff}(\varphi_{\delta_x}(F_{k_i})\cap H),\ b\in\operatorname{aff}(\varphi_{\delta_y}(F_{k_i})\cap H) \big\}.
\]
Without loss of generality, we may assume that $d_1\leq d_2\leq\cdots\leq d_l$. Then when $\epsilon$ is small enough, we have
\[
\operatorname{inf}\big\{ |a-b|:a\in(\varphi_{\delta_x}(\Delta)\cap H)\backslash\operatorname{ri}(\varphi_{\delta_x}(\Delta)\cap H),\ b\in\varphi_{\delta_y}(\Delta)\cap H \big\}= d_1
\]
and we can take $x'\in \varphi_{\delta_x}(F_{k_1})\cap H$ and $y'\in\varphi_{\delta_y}(F_{k_1})\cap H$ such that $|x'-y'|=d_1$. Then $|x'-y'|\leq|x-y|\leq t$, $\chi_{\Delta,\epsilon}^c(x')=\chi_{\Delta,\epsilon}^c(x)=\frac{\delta_x}{\epsilon}$ and $\chi_{\Delta,\epsilon}^c(y')=\chi_{\Delta,\epsilon}^c(y)=\frac{\delta_y}{\epsilon}$. Therefore, we can assume $x,y\in P_i$ for one $i$. 

\ 

By the definition of inner radius \Cref{defn: inner radius}, we have $|z-x_c|\geq \operatorname{inrad}(\Delta)$ for any $z$ in the facets of $\Delta$.

\noindent\textbf{Case 1}. Suppose $x,y,x_c$ are on the same line, and $\chi_{\Delta,\epsilon}^c(x)=\frac{\delta_x}{\epsilon}$ and $\chi_{\Delta,\epsilon}^c(y)=\frac{\delta_y}{\epsilon}$. Then there exists $z\in\Delta\backslash\operatorname{ri}(\Delta)$ such that
\begin{align*}
    &x=x_c+(1-\delta_x)(z-x_c),\\
    &y=x_c+(1-\delta_y)(z-x_c).
\end{align*}
Thus $|x-y|=|z-x_c|\cdot|\delta_x-\delta_y|$ and so
\[
\left| \chi_{\Delta,\epsilon}^c(x)-\chi_{\Delta,\epsilon}^c(y) \right|=\frac{|\delta_x-\delta_y|}{\epsilon}=\frac{|x-y|}{\epsilon|z-x_c|}\leq\frac{t}{\epsilon\cdot\operatorname{inrad}(\Delta)}.
\]

\ 

\noindent\textbf{Case 2}. Suppose $x,y,x_c$ are not on the same line. The Reduction Step allows us to assume that $x,y\in P$ for the truncated hyperpyramid $P$ corresponding to a facet $F$ of $\Delta$. First we suppose that $x\in F$ and $\chi_{\Delta,\epsilon}^c(x)=0$. Let $H$ be the affine $2$-dimensional plane containing~ $x,y,x_c$. Suppose $\chi_{\Delta,\epsilon}^c(y)=\frac{\delta}{\epsilon}$ for $0<\delta<\epsilon$. Then $y\in \varphi_\delta(F)$.

Notice that $F\cap H$ is a line segment, otherwise $F\cap H=\{x\}$ and $\varphi_\delta(F)\cap H=\{y\}$ and thus $x,y,x_c$ are on the same line, a contradiction. It follows from the definition of $\varphi_\epsilon$ and $\varphi_\delta$ that both $\varphi_\epsilon(F)\cap H$ and $\varphi_\delta(F)\cap H$ are line segments. By \Cref{parallel lem: F and phi epsilon F}, the intersections $\varphi_\epsilon(F)\cap H$, $\varphi_\delta(F)\cap H$ and $F\cap H$ are parallel. As in \Cref{figure:construction}, when $\epsilon$ is small enough, we can take $x',x''\in F\cap H$, $y'\in\varphi_\delta(F)\cap H$ and $z\in\varphi_\epsilon(F)\cap H$ such that
\begin{itemize}
    \item $y',z,x_c$ are on the same line,
    \item $|x'-y'|=\operatorname{inf}\{|a-b|:a\in\operatorname{aff}(F\cap H),\ b\in\operatorname{aff}(\varphi_\delta(F)\cap H)\}$,
    \item $|z-x''|=\operatorname{inf}\{|a-b|:a\in\operatorname{aff}(F\cap H),\ b\in\operatorname{aff}(\varphi_\epsilon(F)\cap H)\}$,
    \item the line $\ell_{x'y'}$ passing through $x',y'$ and the line $\ell_{x''z}$ passing through $x'',z$ are parallel.
\end{itemize}

\begin{figure}[H]
    \centering
    \columnwidth=5cm
    \def\svgwidth{\columnwidth}
\begingroup%
  \makeatletter%
  \providecommand\color[2][]{%
    \errmessage{(Inkscape) Color is used for the text in Inkscape, but the package 'color.sty' is not loaded}%
    \renewcommand\color[2][]{}%
  }%
  \providecommand\transparent[1]{%
    \errmessage{(Inkscape) Transparency is used (non-zero) for the text in Inkscape, but the package 'transparent.sty' is not loaded}%
    \renewcommand\transparent[1]{}%
  }%
  \providecommand\rotatebox[2]{#2}%
  \newcommand*\fsize{\dimexpr\f@size pt\relax}%
  \newcommand*\lineheight[1]{\fontsize{\fsize}{#1\fsize}\selectfont}%
  \ifx\svgwidth\undefined%
    \setlength{\unitlength}{399.47560708bp}%
    \ifx\svgscale\undefined%
      \relax%
    \else%
      \setlength{\unitlength}{\unitlength * \real{\svgscale}}%
    \fi%
  \else%
    \setlength{\unitlength}{\svgwidth}%
  \fi%
  \global\let\svgwidth\undefined%
  \global\let\svgscale\undefined%
  \makeatother%
  \begin{picture}(1,0.68265668)%
    \lineheight{1}%
    \setlength\tabcolsep{0pt}%
    \put(0,0){\includegraphics[width=\unitlength,page=1]{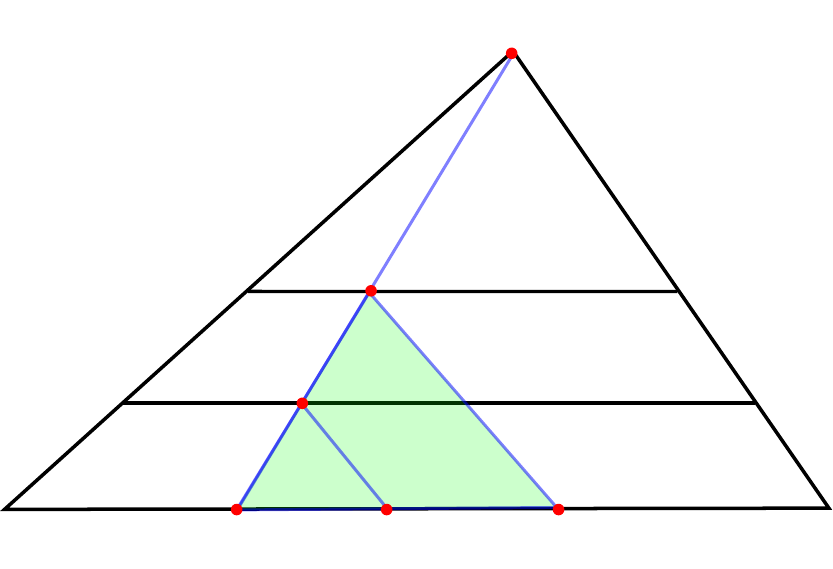}}%
    \put(0.51022085,0.61160646){\color[rgb]{0,0,0}\makebox(0,0)[lt]{\lineheight{1.25}\smash{\begin{tabular}[t]{l}$x_c$\end{tabular}}}}%
    \put(0.3929534,0.35139715){\color[rgb]{0,0,0}\makebox(0,0)[lt]{\lineheight{1.25}\smash{\begin{tabular}[t]{l}$z$\end{tabular}}}}%
    \put(0.65922433,-0.00357832){\color[rgb]{0,0,0}\makebox(0,0)[lt]{\lineheight{1.25}\smash{\begin{tabular}[t]{l}$x''$\end{tabular}}}}%
    \put(0.43330986,-0.00410703){\color[rgb]{0,0,0}\makebox(0,0)[lt]{\lineheight{1.25}\smash{\begin{tabular}[t]{l}$x'$\end{tabular}}}}%
    \put(0.20030483,-0.00572022){\color[rgb]{0,0,0}\makebox(0,0)[lt]{\lineheight{1.25}\smash{\begin{tabular}[t]{l}$x'''$\end{tabular}}}}%
    \put(0.31020004,0.23307196){\color[rgb]{0,0,0}\makebox(0,0)[lt]{\lineheight{1.25}\smash{\begin{tabular}[t]{l}$y'$\end{tabular}}}}%
  \end{picture}%
\endgroup%

    \caption{An example of the truncated hyperpyramid $P\cap H$}
    \label{figure:construction}
\end{figure}
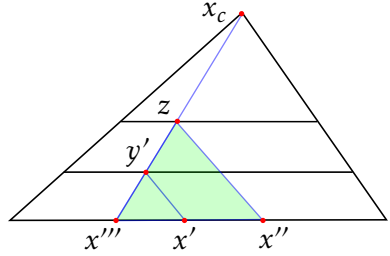

With these points we have
\begin{itemize}
    \item $|x'-y'|\leq|x-y|\leq t$, 
    \item $\chi_{\Delta,\epsilon}^c(x')=\chi_{\Delta,\epsilon}^c(x)=0$, $\chi_{\Delta,\epsilon}^c(y')=\chi_{\Delta,\epsilon}^c(y)=\frac{\delta}{\epsilon}$.
\end{itemize}

Set
\begin{align*}
   &d_F\coloneqq\operatorname{inf} \big\{|a-b|:a\in\operatorname{aff}(F\cap H),\ b\in\operatorname{aff}(\varphi_\epsilon(F)\cap H)\big\},\\
   &D_F\coloneqq\operatorname{inf} \big\{ |a-x_c|:a\in\operatorname{aff}(F\cap H) \big\}.
\end{align*}
Then $|z-x''|=d_F$. Let $x'''$ be the intersection of $F\cap H$ and the ray passing through $y'$ with start point $x_c$, as in \Cref{figure:construction}. Then by similar triangles
\[
\frac{d_F}{D_F}=\frac{|x'''-z|_2}{|x'''-x_c|_2}=\epsilon
\]
and thus $|x''-z|=d_F=\epsilon D_F\geq \epsilon\cdot \operatorname{inrad}(\Delta)$. By \Cref{piecewise affine} and its proof, we have $\chi_{\Delta,\epsilon}^c$ is affine on $P$. Using similar triangles as in \Cref{figure:similar triangles}, we get

\[
\left| \chi_{\Delta,\epsilon}^c(x)-\chi_{\Delta,\epsilon}^c(y) \right|=\left| \chi_{\Delta,\epsilon}^c(x')-\chi_{\Delta,\epsilon}^c(y') \right|=\frac{|x'-y'|_2}{|x''-z|_2}=\frac{|x'-y'|}{|x''-z|}\leq\frac{t}{\epsilon\cdot\operatorname{inrad}(\Delta)},
\]
where $\frac{|x'-y'|_2}{|x''-z|_2}=\frac{|x'-y'|}{|x''-z|}$ holds since $x''',y',z$ are in the same line and $\ell_{x'y'}$ and $\ell_{x''z}$ are parallel.

\begin{figure}[H]
    \centering
    \columnwidth=5.5cm
    \def\svgwidth{\columnwidth}
\begingroup%
  \makeatletter%
  \providecommand\color[2][]{%
    \errmessage{(Inkscape) Color is used for the text in Inkscape, but the package 'color.sty' is not loaded}%
    \renewcommand\color[2][]{}%
  }%
  \providecommand\transparent[1]{%
    \errmessage{(Inkscape) Transparency is used (non-zero) for the text in Inkscape, but the package 'transparent.sty' is not loaded}%
    \renewcommand\transparent[1]{}%
  }%
  \providecommand\rotatebox[2]{#2}%
  \newcommand*\fsize{\dimexpr\f@size pt\relax}%
  \newcommand*\lineheight[1]{\fontsize{\fsize}{#1\fsize}\selectfont}%
  \ifx\svgwidth\undefined%
    \setlength{\unitlength}{386.30373497bp}%
    \ifx\svgscale\undefined%
      \relax%
    \else%
      \setlength{\unitlength}{\unitlength * \real{\svgscale}}%
    \fi%
  \else%
    \setlength{\unitlength}{\svgwidth}%
  \fi%
  \global\let\svgwidth\undefined%
  \global\let\svgscale\undefined%
  \makeatother%
  \begin{picture}(1,0.98402578)%
    \lineheight{1}%
    \setlength\tabcolsep{0pt}%
    \put(0,0){\includegraphics[width=\unitlength,page=1]{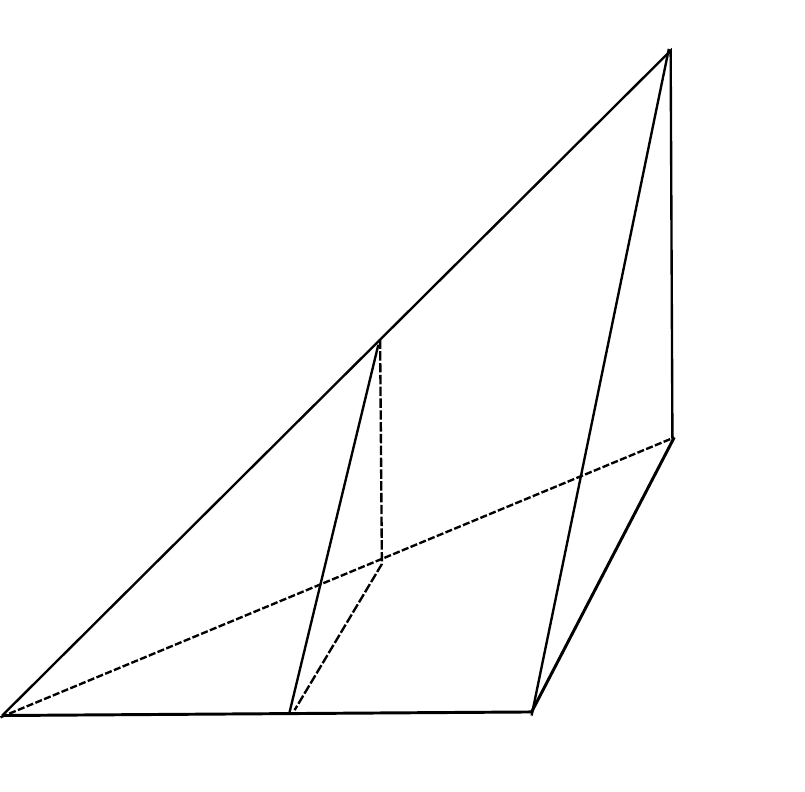}}%
    \put(0.28405471,0.00753336){\color[rgb]{0,0,0}\makebox(0,0)[lt]{\lineheight{1.25}\smash{\begin{tabular}[t]{l}$(x',0)$\end{tabular}}}}%
    \put(0.58526431,0.01196236){\color[rgb]{0,0,0}\makebox(0,0)[lt]{\lineheight{1.25}\smash{\begin{tabular}[t]{l}$(x'',0)$\end{tabular}}}}%
    \put(0.49224403,0.34915376){\color[rgb]{0,0,0}\makebox(0,0)[lt]{\lineheight{1.25}\smash{\begin{tabular}[t]{l}$(y,0)$\end{tabular}}}}%
    \put(0.86263015,0.45048978){\color[rgb]{0,0,0}\makebox(0,0)[lt]{\lineheight{1.25}\smash{\begin{tabular}[t]{l}$(z,0)$\end{tabular}}}}%
    \put(0.19578191,0.60327875){\color[rgb]{0,0,0}\makebox(0,0)[lt]{\lineheight{1.25}\smash{\begin{tabular}[t]{l}$(y,\chi_{\Delta,\epsilon}^c(y))$\end{tabular}}}}%
    \put(0.84743966,0.94544898){\color[rgb]{0,0,0}\makebox(0,0)[lt]{\lineheight{1.25}\smash{\begin{tabular}[t]{l}$(z,1)$\end{tabular}}}}%
    \put(0,0){\includegraphics[width=\unitlength,page=2]{similar_triangulars1.pdf}}%
  \end{picture}%
\endgroup%

    \caption{Similar triangles for computation}
    \label{figure:similar triangles}
\end{figure}

For the general case $x,y\in\Delta\backslash\operatorname{ri}(\Delta_\epsilon)$, the argument is essentially the same and we still get
\[
\left| \chi_{\Delta,\epsilon}^c(x)-\chi_{\Delta,\epsilon}^c(y) \right|\leq\frac{t}{\epsilon\cdot\operatorname{inrad}(\Delta)}.
\]

Therefore,
\[
\rho(\chi_{\Delta,\epsilon}^c,t)=\operatorname{sup}\left\{ 
\left| \chi_{\Delta,\epsilon}^c(x)-\chi_{\Delta,\epsilon}^c(y)\right|: x,y\in\R^d,\ |x-y|\leq t\right\} \leq\frac{t}{\epsilon\cdot\operatorname{inrad}(\Delta)}.\quad\qedsymbol
\]

\subsection{Proof of Koksma's inequality over polytopes}

Now we can prove one of the main tools in this paper, a version of Koksma's inequality over polytopes. 

\begin{thm}\label{Koksma polytope 0}
    Let $f:[0,1]^d\to\R^d$ be a continuous function and let $(x_1,\ldots,x_n)$ be a finite sequence in $[0,1)^d$ with discrepancy $D=D(x_1,\ldots,x_n)$. Let $\Delta\subset[0,1]^d$ be a polytope of dimension $d$. Then for $D$ small enough we have
    \begin{multline*}
        \left|\frac{1}{n}\sum_{\substack{i=1 \\ x_i\in\Delta}}^nf(x_i)-\int_{\Delta} f\d\mu\right|\leq (1+2^{d+1})\rho(f,D^{1/(d+1)})+M\Big( \frac{(1+2^{d+1})D^{1/(2d+2)}}{\operatorname{inrad}(\Delta)}\\
        +
        (4d\sqrt d+1)D^{1/d}e(\Delta)+\frac{2\operatorname{diam}(\Delta)S(\Delta)D^{1/(2d+2)}}{\sqrt{d}} \Big),
    \end{multline*}
    where
    \begin{itemize}
        \item $\rho(f,D^{1/(d+1)})$ is the modulus of continuity of $f$ for $D^{1/(d+1)}$,
        \item $M$ is an upper bound of $|f|$, i.e. $|f(x)|\leq M$ for all $x\in[0,1)^d$,
        \item $e(\Delta)$ is the number of facets of $\Delta$,
        \item $S(\Delta)$ is the surface area of $\Delta$, i.e. it is the sum of the volumes of all facets of $\Delta$.
    \end{itemize}
\end{thm}
\begin{proof}
    We want to apply \Cref{a Koksma in hypercube} to prove this theorem. To do this, we approximate $f\cdot\chi_\Delta$ by a continuous function.
    For $\epsilon>0$, let $\chi_{\Delta,\epsilon}^c$ be the continuous characteristic function we constructed in \Cref{subsection: Construction of continuous characteristic function}. Then we have
    \begin{multline*}
        \left| \frac{1}{n}\sum_{i=1}^n f(x_i)\chi_\Delta(x_i)-\int_{[0,1]^d}f\chi_\Delta\d\mu \right|\\
        \leq 
        \left| \frac{1}{n}\sum_{i=1}^n f(x_i)\chi_{\Delta,\epsilon}^c(x_i)-\int_{[0,1]^d}f\chi_{\Delta,\epsilon}^c\d\mu \right|\\
        + \frac{1}{n}\left| \sum_{i=1}^n\left( f(x_i)\chi_\Delta(x_i) - f(x_i)\chi_{\Delta,\epsilon}^c(x_i) \right) \right|\\
        + \left| \int_{[0,1]^d} \left( f\chi_\Delta - f\chi_{\Delta,\epsilon}^c \right)\d\mu  \right|.
    \end{multline*}
    Let us first estimate the first difference. In order to apply \Cref{a Koksma in hypercube} to it, we only need to understand the modulus of continuity of $f\cdot\chi_{\Delta,\epsilon}^c$. For any two points $x,y\in[0,1]^d$, by the triangle inequality,
    \begin{align*}
        \left|f(x)\chi_{\Delta,\epsilon}^c(x)-f(y)\chi_{\Delta,\epsilon}^c(y)\right|&
        \leq \left|f(x)\chi_{\Delta,\epsilon}^c(x) - f(y)\chi_{\Delta,\epsilon}^c(x) + f(y)\chi_{\Delta,\epsilon}^c(x)-f(y)\chi_{\Delta,\epsilon}^c(y)\right|\\
        &\leq |f(x)-f(y)| + |f(y)|\left|\chi_{\Delta,\epsilon}^c(x)-\chi_{\Delta,\epsilon}^c(y)\right|.
    \end{align*}
    Since $f$ is continuous on $[0,1]^d$, there exists a positive real number $M$ such that $|f(x)|\leq M$ for all $x$. By \Cref{prop:estimate of modulus of continuity}, for sufficiently small $\epsilon$ and $t$, we have
    \[
    \rho\left(f\cdot\chi_{\Delta,\epsilon}^c,t\right)\leq\rho(f,t) + M\cdot\rho(\chi_{\Delta,\epsilon}^c,t)\leq \rho(f,t)+\frac{Mt}{\epsilon\cdot\operatorname{inrad}(\Delta)}.
    \]
    Therefore, we get
    \[
    \left|\frac{1}{n}\sum_{i=1}^n f(x_i)\chi_{\Delta,\epsilon}^c(x_i)-\int_{[0,1]^d}f\chi_{\Delta,\epsilon}^c\d\mu\right|\leq (1+2^{d+1})\left(\rho(f,D^{1/(d+1)}) +\frac{M D^{1/(d+1)}}{\epsilon\cdot\operatorname{inrad}(\Delta)} \right)
    \]
    for small enough $\epsilon$.

    For the second difference, we know that $\chi_\Delta$ and $\chi_{\Delta,\epsilon}^c$ are only different on $\Delta\backslash\operatorname{ri}(\Delta_\epsilon)$. Therefore,
    \begin{align*}
        \frac{1}{n}\left| \sum_{i=1 }^n\left( f(x_i)\chi_\Delta(x_i) - f(x_i)\chi_{\Delta,\epsilon}^c(x_i) \right) \right|
        &\leq\frac{1}{n} \sum_{\substack{i=1 \\ x_i\in\Delta\backslash\operatorname{ri}(\Delta_\epsilon)}}^n |f(x_i)| \left| \chi_\Delta(x_i) - \chi_{\Delta,\epsilon}^c(x_i) \right|\\
        &\leq M \frac{\sum_{i=1}^n\chi_{\Delta\backslash\operatorname{ri}(\Delta_\epsilon)}(x_i)}{n}.
    \end{align*}
    By \Cref{mid: union of polytopes}, $\Delta\backslash\operatorname{ri}(\Delta_\epsilon)$ is a union of polytopes $P_1,\ldots,P_e$ such that $\operatorname{ri}(P_i)\cap\operatorname{ri}(P_j)=\varnothing$ for $i\neq j$, where $e\coloneqq e(\Delta)$ is the number of factes of $\Delta$. Therefore, we can find convex subsets~ $P_1',\ldots,P_e'$ such that
    \begin{itemize}
        \item $\operatorname{ri}(P_i')=\operatorname{ri}(P_i)$,
        \item $P_i'\cap P_j'=\varnothing$ for all $i\neq j$,
        \item $\Delta\backslash\operatorname{ri}(\Delta_\epsilon)=P_1'\cup\cdots\cup P_e'$.
    \end{itemize}
    Let $J(x_1,\ldots,x_n)$ be the isotropic discrepancy of $(x_1,\ldots,x_n)$. It follows from \Cref{isotropical discrepancy and discrepancy} that
    \begin{align*}
        \left| \frac{\sum_{i=1}^n\chi_{\Delta\backslash\operatorname{ri}(\Delta_\epsilon)}(x_i)}{n}-\mu(\Delta\backslash\Delta_\epsilon) \right|
        &\leq \left| \frac{\sum_{i=1}^n\chi_{P_1'}(x_i)}{n}-\mu(P_1') \right| + \cdots + \left| \frac{\sum_{i=1}^n\chi_{P_e'}(x_i)}{n}-\mu(P_e') \right|\\
        &\leq eJ(x_1,\ldots,x_n)\\
        &\leq e(4d\sqrt{d}+1)D^{1/d}.
    \end{align*}
    Thus
    \[
    \frac{\sum_{i=1}^n\chi_{\Delta\backslash\operatorname{ri}(\Delta_\epsilon)}(x_i)}{n}\leq e(4d\sqrt{d}+1)D^{1/d}+\mu(\Delta\backslash\Delta_\epsilon)
    \]
    and so we find an upper bound for the second difference. Similarly, for the third difference, we have
    \begin{align*}
        \left|\int_{[0,1]^d} \left(f\chi_\Delta - f\chi_{\Delta,\epsilon}^c\right)\d\mu \right|&\leq \int_{\Delta\backslash\Delta_\epsilon} |f|\cdot\left|\chi_\Delta-\chi_{\Delta,\epsilon}^c\right|\d\mu\\
        &\leq M\mu(\Delta\backslash\Delta_\epsilon).
    \end{align*}
    By \Cref{volume bet Delta and Delta_epsilon} we have $\mu(\Delta\backslash\Delta_\epsilon)\leq\epsilon\cdot\operatorname{diam}(\Delta)S(\Delta)/\sqrt{d}$. When $D$ is small enough, we can take~ $\epsilon=D^{1/(2d+2)}$ such that $\epsilon$ is also small enough. Combined with all these inequalities, we get the result.
\end{proof}

\begin{rmk}
    Note that \Cref{Koksma polytope 0} does not imply \Cref{a Koksma in hypercube}. If $\Delta=[0,1]^d$, then we can apply \Cref{a Koksma in hypercube} directly to get an upper bound of the difference without the term
    \[
    M\Big( \frac{(1+2^{d+1})D^{1/(2d+2)}}{\operatorname{inrad}(\Delta)}\\
        +
        (4d\sqrt d+1)D^{1/d}e(\Delta)+\frac{2\operatorname{diam}(\Delta)S(\Delta) D^{1/(2d+2)}}{\sqrt{d}} \Big)
    \]
    in \Cref{Koksma polytope 0}.
\end{rmk}

\begin{cor}\label{Koksma polytope}
    Let $f:[0,1]^d\to\R^d$ be a continuous function and let $(x_1,\ldots,x_n)$ be a finite sequence in $[0,1)^d$ with discrepancy $D=D(x_1,\ldots,x_n)$. Let $\Delta\subset[0,1]^d$ be a polytope of dimension~ $d$. Then for $D$ small enough we have
    \[
    \left|\frac{1}{n}\sum_{\substack{i=1 \\ x_i\in\Delta}}^nf(x_i)-\int_{\Delta} f\d\mu\right|\ll_\Delta \rho(f,D^{1/(d+1)}) + M D^{1/(2d+2)},
    \]
    where $M=\operatorname{max}\left\{|f(x)|:x\in[0,1]^d\right\}$.
\end{cor}

\newpage

\section{Proof of the Main Theorem}\label{section proof}

In this section \Cref{main thm} is proved. 

As in \Cref{main thm}, we let $d,k$ be integers, let $P\in\overline{\Q}[T_1^{\pm1},\ldots,T_d^{\pm 1}]\backslash\{0\}$ be an essentially atoral Laurent polynomial with at most $k$ nonzero terms, and let $\Delta\subset[0,1)^d$ be a $d$-dimensional polytope. 

For a torsion point $\omega\in\mathbb{G}_m^d$, according to Laurent's Theorem \cite{Laurent}, if $\delta(\omega)$ is large in terms of $d$ and $P$, then $P(\omega^\sigma)\neq 0$ for all $\sigma\in\operatorname{Gal}(\Q(\omega)/\Q)$. We fix a torsion point $\omega\in\mathbb{G}_m^d$ with~ $\delta(\omega)$ sufficiently large such that $P(\omega^\sigma)\neq 0$ for all $\sigma\in\operatorname{Gal}(\Q(\omega)/\Q)$. Denote $n\coloneqq[\Q(\omega):\Q]$. For the Galois orbit of $\omega$, we get a finite set $\{x_1,\ldots,x_n\}\subset[0,1)^d$ such that
\[
\{e(x_1),\ldots,e(x_n)\}=\{\omega^\sigma:\sigma\in\operatorname{Gal}(\Q(\omega)/\Q)\}.
\]
Then $P(e(x_i))\neq 0$ for $1\leq i\leq n$. Let $D=D(x_1,\ldots,x_n)$ (resp. $J=J(x_1,\ldots,x_n)$) be the discrepancy (resp. isotropic discrepancy) of the sequence $(x_1,\ldots,x_n)$. Our goal is to estimate the difference
\[
\left| \frac{1}{n}\sum_{\substack{i=1 \\ x_i\in\Delta}}^n\operatorname{log}|P(e(x_i))|-\int_\Delta\operatorname{log}|P(e(x))|\d \mu(x) \right|.
\]

We first examine the easiest case that $P$ is a constant. In this case,
\[
\left| \frac{1}{n}\sum_{\substack{i=1 \\ x_i\in\Delta}}^n\operatorname{log}|P(e(x_i))|-\int_\Delta\operatorname{log}|P(e(x))|\d \mu(x) \right|\ll_P \left| \frac{\sum_{i=1}^n\chi_\Delta(x_i)}{n}-\mu(\Delta)\right|\leq J
\]
by the definition of $J$. By \Cref{isotropical discrepancy and discrepancy} and \Cref{discrepancy over Galois orbit}, we have
\[
J\ll_d D^{1/d}\ll_d \left(\frac{(\operatorname{log}\delta(\omega))^d}{\delta(\omega)^{1/2}}\right)^{1/d}\ll_d \delta(\omega)^{-1/(4d)}
\]
as $\delta(\omega)\to\infty$. The desired result follows.

Now suppose $P\in\overline{\Q}[T_1^{\pm1},\ldots,T_d^{\pm 1}]\backslash\overline{\Q}$. We may assume that $P$ is a polynomial, as multiplying by a monomial does not change the difference between the discrete sum and the integral stated above. Let $c>0$ be the largest maximum norm among the coefficients of $P$. Then there exists a polynomial $Q\in\overline{\Q}[T_1,\ldots,T_d]\backslash\overline{\Q}$ such that $P=cQ$ with at most $k$ non-zero terms and the maximum norm of the coefficient vector of $Q$ is $1$. Hence,
\begin{multline*}
    \left| \frac{1}{n}\sum_{\substack{i=1 \\ x_i\in\Delta}}^n\operatorname{log}|P(e(x_i))|-\int_\Delta\operatorname{log}|P(e(x))|\d \mu(x) \right|\\
    \leq \left| \frac{1}{n}\sum_{\substack{i=1 \\ x_i\in\Delta}}^n\operatorname{log}c - \int_\Delta\operatorname{log}c\d \mu \right| + \left| \frac{1}{n}\sum_{\substack{i=1 \\ x_i\in\Delta}}^n\operatorname{log}|Q(e(x_i))|-\int_\Delta\operatorname{log}|Q(e(x))|\d \mu(x) \right|.
\end{multline*}
Note that the first difference above can also be estimated by isotropic discrepancy:
\[
\left| \frac{1}{n}\sum_{\substack{i=1 \\ x_i\in\Delta}}^n\operatorname{log}c - \int_\Delta\operatorname{log}c\d \mu \right|=|\operatorname{log}c|\left|\frac{\sum_{i=1}^n\chi_{\Delta}(x_i)}{n}-\mu(\Delta)\right|\leq|\operatorname{log}c|J\ll_P J\ll_d\delta(\omega)^{-1/(4d)}.
\]
Therefore,
\begin{multline}\label{pf of main thm: P and Q}
    \left| \frac{1}{n}\sum_{\substack{i=1 \\ x_i\in\Delta}}^n\operatorname{log}|P(e(x_i))|-\int_\Delta\operatorname{log}|P(e(x))|\d \mu(x) \right|\\
    \ll_P \delta(\omega)^{-1/(4d)} + \left| \frac{1}{n}\sum_{\substack{i=1 \\ x_i\in\Delta}}^n\operatorname{log}|Q(e(x_i))|-\int_\Delta\operatorname{log}|Q(e(x))|\d \mu(x) \right|.
\end{multline}

Suppose $k\geq 2$. We set $\operatorname{log}_r(x)=\operatorname{log}\operatorname{max}\{r,x\}$ for $r>0$. By \Cref{discrepancy over Galois orbit}, $\delta(\omega)$ is large implies that $D$ is small. At the end of proof, we will take $r$ in terms of $D$ such that the small~ $D$ implies that $r$ is also small.

By the triangle inequality,

\begin{multline}\label{eq:tri ineq}
    \left| \frac{1}{n}\sum_{\substack{i=1 \\ x_i\in\Delta}}^n\operatorname{log}|Q(e(x_i))|-\int_\Delta\operatorname{log}|Q(e(x))|\d \mu(x) \right|\\
    \leq \left| \frac{1}{n}\sum_{\substack{i=1 \\ x_i\in\Delta}}^n\operatorname{log}_r|Q(e(x_i))|-\int_\Delta \operatorname{log}_r|Q(e(x))|\d \mu(x) \right|\\
    +\frac{1}{n}\left| \sum_{i=1}^n\big( \operatorname{log}|Q(e(x_i))| - \operatorname{log}_r|Q(e(x_i))|\big) \right|\\
    +\int_\Delta\big| \operatorname{log}|Q(e(x))| - \operatorname{log}_r|Q(e(x))| \big| \d \mu(x) .
\end{multline}

For the first difference, as $D$ is small, we can apply \Cref{Koksma polytope} to it. In order to apply \Cref{Koksma polytope}, we need to estimate the maximum value of $\operatorname{log}_r|Q(e(x))|$ on $[0,1]^d$. As~ $|Q(e(x))|$ is continuous on the compact set $[0,1]^d$, we have that $|Q(e(x))|$ is bounded. Taking $r$ sufficiently small, we have $\big|\operatorname{log}_r|Q(e(x))|\big|\leq|\operatorname{log}r|$ for all $x\in[0,1]^d$. By \Cref{Koksma polytope},

\[
\left| \frac{1}{n}\sum_{\substack{i=1 \\ x_i\in\Delta}}^n\operatorname{log}_r|Q(e(x_i))|-\int_\Delta \operatorname{log}_r|Q(e(x))|\d \mu(x) \right|
    \ll_\Delta \rho(\operatorname{log}_r|Q(e(x))|;D^{1/(d+1)})+(-\operatorname{log}r)D^{1/(2d+2)}
\]
for $D$ small enough. The modulus of continuity of the function $\operatorname{log}_r|Q(e(x))|$ can be estimated by \cite[Lemma 7.5]{DimitrovHabegger}, which states that 
\[
\rho(\operatorname{log}_r|Q(e(x))|;t)\ll_{d,k}\frac{\operatorname{deg(Q)t}}{r}
\]
for any $r\in(0,1]$ and $t>0$. In our setting, we conclude that
\[
\rho(\operatorname{log}_r|Q(e(x))|;D^{1/(d+1)})\ll_{P}\frac{D^{1/(d+1)}}{r}
\]
and thus
\begin{equation}\label{eq:first difference}
    \left| \frac{1}{n}\sum_{\substack{i=1 \\ x_i\in\Delta}}^n\operatorname{log}_r|Q(e(x_i))|-\int_\Delta \operatorname{log}_r|Q(e(x))|\d \mu(x) \right|\ll_{\Delta,Q}\frac{D^{1/(d+1)}}{r} + (-\operatorname{log}r)D^{1/(2d+2)}.
\end{equation}

For the second difference, the triangle inequality gives
\[
\frac{1}{n}\left| \sum_{i=1}^n \big( \operatorname{log}|Q(e(x_i))| - \operatorname{log}_r|Q(e(x_i))| \big) \right|\\
    \leq\frac{1}{n}\sum_{\substack{i=1 \\ |Q(e(x_i))|<r}}^n\big| \operatorname{log}_r|Q(e(x_i))| \big| + \frac{1}{n}\sum_{\substack{i=1 \\ |Q(e(x_i))|<r}}^n\big| \operatorname{log}|Q(e(x_i))| \big|.
\]
By definition, $\operatorname{log}_r|Q(e(x))|=\operatorname{log}r$ if $|Q(e(x))|<r$. Hence, 
\[
\frac{1}{n}\sum_{\substack{i=1 \\ |Q(e(x_i))|<r}}^n\big| \operatorname{log}_r|Q(e(x_i))| \big|=\frac{\sharp\{i:|Q(e(x_i))|<r\}\ |\operatorname{log}r|}{n}.
\]
This can be further estimated using \cite[Lemma 7.4]{DimitrovHabegger}, which indicates that
\[
\frac{\sharp\left\{i:|Q(e(x_i))|\leq r\right\}}{n}\ll_{d,k} r^{1/(2k)}+\operatorname{deg}(Q)\frac{D^{1/(d+1)}}{r}.
\]
Thus,
\[
\frac{1}{n}\sum_{\substack{i=1 \\ |Q(e(x_i))|<r}}^n\big| \operatorname{log}_r|Q(e(x_i))| \big|\ll_Q(-\operatorname{log}r)\left(r^{1/(2k)}+\frac{D^{1/(d+1)}}{r}\right).
\]
Note that \cite[Lemma 7.7]{DimitrovHabegger} shows
\[
\frac{1}{n}\sum_{\substack{i=1 \\ |Q(e(x_i))|<r}}^n\big| \operatorname{log}|Q(e(x_i))| \big|\ll_{d,k}\frac{(\operatorname{deg}Q)D^{1/(d+1)}}{r^2}+r^{1/(4k)}+\left| m(Q)-\frac{1}{n}\sum_{i=1}^n\operatorname{log}|Q(e(x_i))| \right|.
\]
We conclude from the main theorem of \cite{DimitrovHabegger}, as stated in \Cref{DH main thm}, that
\[
\left| m(Q)-\frac{1}{n}\sum_{i=1}^n\operatorname{log}|Q(e(x_i))| \right|\ll_Q \delta(\omega)^{-\gamma}
\]
for a constant $\gamma=\gamma(d,k)>0$. Putting the above inequalities together, for $r$ small enough, we get
\begin{equation}\label{eq:second difference}
    \frac{1}{n}\left| \sum_{i=1}^n \operatorname{log}|Q(e(x_i))| - \operatorname{log}_r|Q(e(x_i))| \right|\ll_Q \delta(\omega)^{-\gamma} + r^{1/(4k)} + \frac{D^{1/(d+1)}}{r^2}.
\end{equation}

Let $\mathcal{S}_Q(r)=\left\{x\in[0,1)^d:|Q(e(x))|<r\right\}$. Similarly, for the third difference, we have
\[
\int_\Delta\big| \operatorname{log}|Q(e(x))| - \operatorname{log}_r|Q(e(x))| \big| \d \mu(x)
    \leq\int_{\mathcal{S}_Q(r)}\big| \operatorname{log}_r|Q(e(x))| \big| \d \mu(x) + \int_{\mathcal{S}_Q(r)}\big| \operatorname{log}|Q(e(x))| \big| \d \mu(x).
\]
By \cite[Lemma A.3(i)]{DimitrovHabegger}, the Lebesgue volume of $\mathcal{S}_Q(r)$ is $\ll_{d,k} r^{1/(2k-2)}$. Applying this and using $\operatorname{log}_r|Q(e(x))|=\operatorname{log}r$ if $|Q(e(x))|<r$, we get
\[
\int_{\mathcal{S}_Q(r)}\big| \operatorname{log}_r|Q(e(x))| \big| \d x\leq (-\operatorname{log}r)\cdot\mu(\mathcal{S}_Q(r))\ll_{d,k}(-\operatorname{log}r)r^{1/(2k-2)}.
\]
Moreover, it is shown in \cite[Lemma A.4]{DimitrovHabegger} that
\[
\int_{\mathcal{S}_Q(r)}\big| \operatorname{log}|Q(e(x))| \big| \d\mu (x)\ll_{d,k}r^{1/(4k-4)},
\]
and so for $r$ small enough we have
\begin{equation}\label{eq:third difference}
    \int_\Delta\big| \operatorname{log}|Q(e(x))| - \operatorname{log}_r|Q(e(x))| \big| \d\mu(x)\ll_{d,k} r^{1/(4k-4)}.
\end{equation}

Choose $r=D^{1/(4d+4)}$ now. Then gathering the inequalities (\ref{eq:tri ineq}), (\ref{eq:first difference}), (\ref{eq:second difference}) and (\ref{eq:third difference}) together we have
\[
\left| \frac{1}{n}\sum_{\substack{i=1 \\ x_i\in\Delta}}^n\operatorname{log}|Q(e(x_i))|-\int_\Delta\operatorname{log}|Q(e(x))|\d \mu(x) \right|\ll_{\Delta,Q} \delta(\omega)^{-\gamma} + D^{1/(16k(d+1))}
\]
if $\delta(\omega)$ is sufficiently large. Recall that $P=cQ$ for $c>0$. The conclusion follows from (\ref{pf of main thm: P and Q}) and \Cref{discrepancy over Galois orbit}. \qedsymbol

\begin{rmk}\label{rmk: kappa}
    An explicit value of $\kappa$ in \Cref{main thm} can be deduced from this proof. Indeed, by \Cref{discrepancy over Galois orbit},
    \[
    D^{1/(16k(d+1))}\ll_d\left(\frac{(\operatorname{log}\delta(\omega))^d}{\delta(\omega)^{1/2}}\right)^{1/(16k(d+1))}\ll_d \left( \frac{\delta(\omega)^{1/4}}{\delta(\omega)^{1/2}}
 \right)^{1/(16k(d+1))}=\delta(\omega)^{-1/(64k(d+1))}.
    \]
    Therefore, one can choose
    \[
\kappa=\operatorname{min}\left\{\gamma, \frac{1}{64k(d+1)} \right\}
\]
and the proof of \cite[Theorem 8.8]{DimitrovHabegger} contains a way to determine $\gamma=\gamma(d,k)$. This approach is summarized in \Cref{Appendix estimate log} and presented as \Cref{algorithmkappa}.
\end{rmk}

\newpage

\section{Application}\label{section application}

In this section, we answer \cite[Question 6.2]{RM}, posed by Gualdi and Sombra, by applying \Cref{main thm}. Let us introduce the background of this question. The paper \cite{RM} studies the height of the intersection of
\[
L=Z(x_0+x_1+x_2)\subset\mathbb{P}^2(\overline{\Q})
\]
with its translate $\omega L$ by a torsion point $\omega=(\omega_1,\omega_2)\in\mathbb{G}_m^2(\overline{\Q})\cong(\overline{\Q}^{\times})^2$, where
\[
\omega L=Z(x_0+\omega_1^{-1}x_1+\omega_2^{-1}x_2).
\]
When $\omega$ is nontrivial, i.e. $\omega\neq(1,1)$, the intersection $L\cap\omega L$ consists of the point
\[
\mathcal{P}(\omega)=[\omega_2^{-1}-\omega_1^{-1}:1-\omega_2^{-1}:\omega_1^{-1}-1]\in\mathbb{P}^2(\overline{\Q}).
\]
Let us also denote by $\mathcal{P}(\omega)$ the vector of homogeneous coordinates $(\omega_2^{-1}-\omega_1^{-1},1-\omega_2^{-1},\omega_1^{-1}-1)$. We fix an embedding $\iota:\overline{\Q}\hookrightarrow\C$. Let $d=\operatorname{ord}(\omega)$. Then \cite[Lemma 2.3]{RM} gives an explicit formula for the archimedean height of $\mathcal{P}(\omega)$:
\[
\mathrm{h}_\infty(\mathcal{P}(\omega))=\frac{1}{\phi(d)}\sum_{k\in(\Z/d\Z)^\times}\operatorname{log}\operatorname{max}(|\iota(\omega_2^k)-\iota(\omega_1^k)|,|\iota(\omega_2^k)-1|,|\iota(\omega_1^k)-1|),
\]
where $\infty$ is the unique archimedean place of $\Q$ and $|\cdot|$ denotes the usual absolute value on~ $\C$.

We will later consider a sequence in $\mathbb{G}_m^2(\overline{\Q})$ that avoids all proper algebraic subgroups of $\mathbb{G}_m^2(\overline{\Q})$. According to \cite[Chapter 3]{Bombieri_Gubler_2006}, over any field $F$ of characteristic $0$, a one-dimensional algebraic subgroup of $\mathbb{G}_m^2(F)$ is of the form
\begin{equation}\label{notation of algebraic subgroup}
    H_a\coloneqq\{\omega\in\mathbb{G}_m^2(F):\omega^a=1\}
\end{equation}
for an $a\in\Z^2\backslash\{0\}$. We say a sequence $(\omega_\ell)_{\ell\geq 1}$ in $\mathbb{G}^2_m(\overline{\Q})$ is \textit{strict} if for all $a=(a_1,a_2)\in\Z^2\backslash\{0\}$, there exists an integer $\ell_0>0$ such that $\omega_\ell$ does not lie in $H_a$ for $\ell>\ell_0$, which means
\[
\omega_{\ell,1}^{a_1}\omega_{\ell,2}^{a_2}\neq 1,\quad\forall\ell>\ell_0,
\]
where $\omega_\ell=(\omega_{\ell,1},\omega_{\ell,2})$. Note that for a sequence $(\omega_\ell)_{\ell\geq 1}$ in $\mathbb{G}_m^2(\overline{\Q})$, the strictness degree~ $\delta(\omega_\ell)\to\infty$ as $\ell\to\infty$ if and only if $(\omega_\ell)_{\ell\geq 1}$ is strict.

Let $(\omega_\ell)_{\ell\geq1}$ be a strict sequence of non-trivial torsion points in $\mathbb{G}_m^2(\overline{\Q})$ now. It is shown in \cite[Proposition 4.1]{RM} that the limit of archimedean heights $\mathrm{h}_\infty(\mathcal{P}(\omega_\ell))$ is an integral by equidistribution theorems. Indeed, fix an integer $\ell$ and write $n_\ell=\phi(\operatorname{ord}(\omega_\ell))$ for simplicity, set
\[
\left\{ e(\alpha_{\ell,1}),\ldots,e(\alpha_{\ell,n_\ell}) \right\}\coloneqq\left\{\iota(\omega_\ell^k):k\in(\Z/\operatorname{ord}(\omega_\ell)\Z)^\times\right\},
\]
where $e(\alpha_{\ell,k})=(e^{\mathrm{i} 2\pi x_{\ell,k}},\ e^{\mathrm{i} 2\pi y_{\ell,k}})$ for each $\alpha_{\ell,k}=(x_{\ell,k},y_{\ell,k})\in[0,1)^2$. Let
\begin{align*}
    f:[0,1)^2&\to\R\cup\{-\infty\},\\
    (x,y)&\mapsto\operatorname{log}\operatorname{max}\left\{ 
\left|e^{\mathrm{i} 2\pi(x-y)}-1\right|,\ \left|e^{\mathrm{i} 2\pi x}-1\right|,\ \left|e^{\mathrm{i} 2\pi y}-1\right| \right\}.
\end{align*}

Then we have 
\[
\lim\limits_{\ell\to\infty}\mathrm{h}_\infty\left( \mathcal{P}(\omega_\ell) \right)=\lim\limits_{\ell\to\infty}\frac{1}{n_\ell}\sum_{k=1}^{n_\ell}f(\alpha_{\ell,k})=\int_{[0,1)^2}f\d\mu.
\]

In \cite[Question 6.2]{RM}, a quantitative version of this equidistribution result is requested. The following is our solution.

Let $\omega\in\mathbb{G}_m^2(\overline{\Q})$ be a non-trivial torsion point. For the Galois orbit of $\omega$, we get a finite set~ $\{\alpha_1,\ldots,\alpha_n\}\subset[0,1)^2$ such that
\[
\left\{ e(\alpha_{1}),\ldots,e(\alpha_{n}) \right\}\coloneqq\left\{\iota(\omega^k):k\in(\Z/\operatorname{ord}(\omega)\Z)^\times\right\},
\]
where $n=\phi(\operatorname{ord}(\omega))$ and $\alpha_k=(x_k,y_k)\in[0,1)^2$. Let us estimate $\left|h(\mathcal{P}(\omega))-\int_{[0,1)^2}f\d\mu\right|$.

As in \cite[Section 5]{RM}, the square $[0,1)^2$ has a partition consisting of $12$ closed triangles as in \Cref{fig:partition}. Among these $12$ triangles, in $4$ of them, $f(x,y)=\operatorname{log}|e^{i 2\pi x}-1|$; in $4$ of them,~ $f(x,y)=\operatorname{log}\left|e^{i 2\pi y}-1\right|$; in the remaining $4$ triangles, $f(x,y)=\operatorname{log}\left|e^{i 2\pi(x-y)}-1\right|$. The first (resp. second / third) group of triangles are marked by green (resp. blue / red) in the figure. We denote these $12$ subsets as $\Omega_{ij}$ for $1\leq i\leq 3$ and $1\leq j\leq 4$, where $\Omega_{11},\ldots,\Omega_{14}$ correspond to green triangles, $\Omega_{21},\ldots,\Omega_{24}$ correspond to blue triangles and $\Omega_{31},\ldots,\Omega_{34}$ correspond to red triangles.

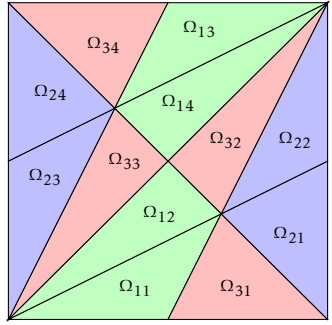
\begin{figure}[H]
    \centering
        \begin{tikzpicture}[scale=0.70]

    \begin{scope}[xshift=-25mm]
    \draw[fill=green!25] (3,0) -- (0,0) -- (4,2) -- cycle; \node at (2.4,0.6) {$\scriptscriptstyle \Omega_{11}$};
    \draw[fill=red!25] (3,0) -- (6,0) -- (4,2)  -- cycle; \node at (4.3,0.6) {$\scriptscriptstyle \Omega_{31}$};
    \draw[fill=blue!25]  (6,0) -- (4,2) -- (6,3)  -- cycle; \node at (5.3,1.6) {$\scriptscriptstyle \Omega_{21}$};                                
    \draw[fill=blue!25]  (6,3) -- (4,2) -- (6,6)  -- cycle; \node at (5.4,3.4) {$\scriptscriptstyle\Omega_{22}$};                                
    \draw[fill=red!25]  (4,2) -- (6,6) -- (3,3) -- cycle; \node at (4.1,3.4) {$\scriptscriptstyle\Omega_{32}$};                                
    \draw[fill=green!25]  (4,2) --  (0,0) -- (3,3) -- cycle; \node at (2.85,2) {$\scriptscriptstyle\Omega_{12}$};
    \draw[fill=red!25] (3,3) -- (0,0) -- (2,4) -- cycle; \node at (2.2,3) {$\scriptscriptstyle\Omega_{33}$};
    \draw[fill=blue!25]  (0,0) -- (2,4) -- (0,3) -- cycle; \node at (0.7,2.7) {$\scriptscriptstyle \Omega_{23}$};
    \draw[fill=blue!25]  (0,3) -- (2,4) -- (0,6) -- cycle; \node at (0.8,4.3) {$\scriptscriptstyle \Omega_{24}$};
    \draw[fill=red!25]  (0,6) -- (2,4) -- (3,6) -- cycle; \node at (1.8,5.2) {$\scriptscriptstyle \Omega_{34}$};
    \draw[fill=green!25] (6,6) -- (2,4) --  (3,6)  -- cycle; \node at (3.6,5.5) {$\scriptscriptstyle \Omega_{13}$};
    \draw[fill=green!25]  (2,4) -- (3,3) -- (6,6) -- cycle; \node at (3.2,4.1) {$\scriptscriptstyle \Omega_{14}$};
    \end{scope}
    \end{tikzpicture}
    \caption{The partition of $[0,1]^2$ into triangles}
    \label{fig:partition}
\end{figure}

Furthermore, for any point in the boundaries of these triangles, its image under the complex exponential map
\[
[0,1)^2\to\mathbb{G}_m^2(\C),\quad (x,y)\mapsto(e^{\mathrm{i}2\pi x},e^{\mathrm{i}2\pi y})
\]
lies in one of the algebraic subgroups $H_{(1,0)},H_{(0,1)},H_{(1,1)},H_{(1,-1)},H_{(2,-1)},H_{(1,-2)}$, as defined by the notation $H_a$ in (\ref{notation of algebraic subgroup}). Therefore, if $\delta(\iota(\omega^k))\geq 3$, then none of the points $\alpha_1,\ldots,\alpha_n$ lie on the boundaries of these $12$ triangles. It follows that
\[
f(\alpha_k)=\sum_{\substack{1\leq i\leq 3 \\ 1\leq j\leq 4}}f\cdot\chi_{\Omega_{ij}}(\alpha_k),\quad\forall 1\leq k\leq n
\]
if $\delta(\omega)\geq 3$, and thus
\begin{equation}\label{ineq: f leq sum over ij}
    \left| \frac{1}{n}\sum_{k=1}^nf(\alpha_k)-\int_{[0,1)^2}f\d\mu \right|\leq\sum_{\substack{1\leq i\leq 3 \\ 1\leq j\leq 4}}\left| \frac{1}{n}\sum_{k=1}^nf\cdot\chi_{\Omega_{ij}}(\alpha_k)-\int_{[0,1)^2}f\cdot\chi_{\Omega_{ij}}\d\mu \right|
\end{equation}
if $\delta(\omega)\geq 3$.

Set $P_1(T_1,T_2)=T_1-1$, $P_2(T_1,T_2)=T_2-1$ and $P_3(T_1,T_2)=T_1-T_2$. In $\Omega_{11},\ldots,\Omega_{14}$ we have $f(x,y)=\operatorname{log}|P_1(e(x,y))|$; in $\Omega_{21},\ldots,\Omega_{24}$ we have $f(x,y)=\operatorname{log}|P_2(e(x,y))|$; in $\Omega_{31},\ldots,\Omega_{34}$ we have $f(x,y)=\operatorname{log}|P_3(e(x,y))|$. The three polynomials $P_1,\ P_2$ and $P_3$ are all binomials and, therefore, are essentially atoral by \Cref{P i are essentially atoral}. Hence, by \Cref{main thm}, if $\delta(\omega)$ is sufficiently large, then there exists a constant $\kappa>0$ such that
\begin{equation}\label{ineq: sum over ij}
    \left| \frac{1}{n}\sum_{k=1}^nf\cdot\chi_{\Omega_{ij}}(\alpha_k)-\int_{[0,1)^2}f\cdot\chi_{\Omega_{ij}}\d\mu \right| \ll \delta(\omega)^{-\kappa}
\end{equation}
for all $1\leq i\leq 3$ and $1\leq j\leq 4$, where the constant is specified by $P_1,P_2,P_3$ and the choice of~ $\Omega_{ij}$. Moreover, by \Cref{rmk: kappa}, we can choose
\[
\kappa=\operatorname{min}\left\{\gamma(2,2),\frac{1}{64\times 2\times 3}\right\},
\]
where $\gamma(2,2)$ can be calculated to be $1/(2^{61}\times 5^5)$ by \Cref{algorithmkappa} with the input $\gamma(1,2)=\gamma(1,4)=1/2$. Thus, $\kappa=1/(2^{61}\times 5^5)$. It follows from (\ref{ineq: f leq sum over ij}) and (\ref{ineq: sum over ij}) that
\[
\left| \frac{1}{n}\sum_{k=1}^nf(\alpha_k)-\int_{[0,1)^2}f\d\mu \right|\ll \delta(\omega)^{-1/(2^{61}\times 5^5)}.
\]
This provides the convergence speed of the archimedean part of the height of the point $\mathcal{P}(\omega)$ as $\delta(\omega)\to\infty$ and answers \cite[Question 6.2]{RM}. Moreover, it is shown in \cite[Proposition 6.1]{RM} that $\int_{[0,1)^2}f\d\mu=\frac{2\zeta(3)}{3\zeta(2)}$, where~ $\zeta$ is the Riemann zeta function. The non-archimedean part of the height of $\mathcal{P}(\omega)$ was given in \cite[Corollary 3.4]{RM}, which states that
\[
\sum_{v_p\in M_\Q\backslash\{\infty\}}\mathrm{h}_{v_p}(\mathcal{P}(\omega))=-\frac{\Lambda(\operatorname{ord}(\omega))}{\phi(\operatorname{ord}(\omega))},
\]
where $\Lambda$ is the Von Mangoldt function. More precisely, if $\operatorname{ord}(\omega)$ is not a prime power, then the non-archimedean part of the height of $\mathcal{P}(\omega)$ is $0$; if $\operatorname{ord}(\omega)=p^e$ for a prime number $p$ and a positive integer $e$, then
\[
\sum_{v_p\in M_\Q\backslash\{\infty\}}\mathrm{h}_{v_p}(\mathcal{P}(\omega))=-\frac{\operatorname{log}p}{p^{e-1}(p-1)}.
\]
In the case where $d\coloneqq\operatorname{ord}(\omega)=p^e$, we have
\[
\left|\sum_{v_p\in M_\Q\backslash\{\infty\}}  \mathrm{h}_{v_p}(\mathcal{P}(\omega))\right|=  \frac{\operatorname{log}p}{p^{e-1}(p-1)}=\frac{p}{e(p-1)}\frac{\operatorname{log}d}{d}\leq\frac{2\operatorname{log}d}{d}.
\]
Since $\delta(\omega)\leq d$ and the function $x\mapsto\frac{\operatorname{log}x}{x}$ decreases as $x\to\infty$, we get
\[
\left| \sum_{v_p\in M_\Q\backslash\{\infty\}}  \mathrm{h}_{v_p}(\mathcal{P}(\omega)) \right| \ll \frac{2\operatorname{log}\delta(\omega)}{\delta(\omega)} \ll \delta(\omega)^{-1/2}
\]
and so $\sum_{v_p\in M_\Q\backslash\{\infty\}}  \mathrm{h}_{v_p}(\mathcal{P}(\omega))\to 0$ as $\operatorname{\delta}(\omega)\to\infty$. In summary, if $\delta(\omega)$ is sufficiently large, then we have
\[
\left| \mathrm{h}(\mathcal{P}(\omega)) - \frac{2\zeta(3)}{3\zeta(2)} \right|\ll \delta(\omega)^{-1/(2^{61}\times 5^5)}
\]
as $\delta(\omega)\to\infty$. Therefore, if $(\omega_\ell)_{\ell\geq 1}$ is a strict sequence of non-trivial torsion points in $\mathbb{G}_m^d(\overline{\Q})$, then
\[
    \left| \mathrm{h}(\mathcal{P}(\omega_\ell)) - \frac{2\zeta(3)}{3\zeta(2)} \right|\ll \delta(\omega_\ell)^{-1/(2^{61}\times 5^5)}
\]
as $\ell\to\infty$, which proves \Cref{main application}.

\begin{rmk}
    Given a strict sequence $(\omega_\ell)_{\ell\geq 1}$ of non-trivial torsion points in a proper algebraic subgroup of $\mathbb{G}_m^2$, we can also consider the limit of $\mathrm{h}(\mathcal{P}(\omega_\ell))$ as $\ell\to\infty$. We refer to \cite[Section 7]{RM} for more details. In this case, the difference between the discrete sum and the integral can be estimated by $\delta(\omega_\ell)^{-1/2^7}$.
\end{rmk}

\newpage

\appendix

\section{Explicit values in the Galois equidistribution theorem}\label{Appendix estimate log}

In the statement of \Cref{main thm}, a constant $\kappa>0$ is involved as the power of $\delta(\omega)^{-1}$. As we mentioned in \Cref{rmk: kappa}, an explicit value of $\kappa$ can be computed and this value is completely determined by two integers $d,k$, where $d$ is the number of variables of $P$ and $k$ is the number of non-zero terms in $P$. This appendix contains an algorithm that helps us to compute $\kappa(d,k)$.

In this appendix, we denote torsion points by $\zeta$ instead of $\omega$ to be consistent with \cite{DimitrovHabegger}. First, let us recall the main theorem of \cite{DimitrovHabegger}, which is a corollary of \cite[Theorem 8.8]{DimitrovHabegger}.

\begin{thm}\label{DH main thm}
    Let $d,k$ be integers, and let $P\in\overline{\Q}[T_1^{\pm 1},\ldots,T_d^{\pm 1}]\backslash\{0\}$ be essentially atoral with at most $k$ nonzero terms. Then there exists a constant $\gamma=\gamma(d,k)>0$ such that the following holds.

    Given any torsion point $\zeta\in\mathbb{G}_m^d$ suppose that the strictness degree $\delta(\zeta)$ is large enough. Then~ $P(\zeta^\sigma)\neq 0$ for all $\sigma\in\operatorname{Gal}(\Q(\zeta)/\Q)$, and as $\delta(\zeta)\to\infty$ we have
    \[
    \left| \left(\frac{1}{[\Q(\zeta):\Q]}\sum_{\sigma\in\operatorname{Gal}(\Q(\zeta)/\Q)}\operatorname{log}|P(\zeta^\sigma)|\right)-m(P)\right|\ll_P \delta(\zeta)^{-\gamma}.
    \]
\end{thm}

For $d=1$, by \cite[Proposition 4.5]{DimitrovHabegger}, we have
\[
    \left| \left(\frac{1}{[\Q(\zeta):\Q]}\sum_{\sigma\in\operatorname{Gal}(\Q(\zeta)/\Q)}\operatorname{log}|P(\zeta^\sigma)|\right)-m(P)\right|\ll_P\frac{(\operatorname{log}\delta(\zeta))^3 \tau(\delta(\zeta))}{\delta(\zeta)},
\]
where $\tau(\delta(\zeta))$ is the number of divisors of the positive integer $\delta(\zeta)$. Thus one can choose $\gamma(1,k)=1-\epsilon$ for any $\epsilon>0$. For $d\geq 2$, the proof of \cite[Theorem 8.8]{DimitrovHabegger} provides a way to find an explicit value of $\gamma=\gamma(d,k)$. Also, in \Cref{DH main thm}, we note that the convergence speed of the equidistribution theorem can be estimated for sufficiently large $\delta(\zeta)$. It is worth to point out that the proof of \cite[Theorem 8.8]{DimitrovHabegger} also allows us to explicitly determine how large $\delta(\zeta)$ should be for certain specific Laurent polynomials $P$. This estimation of $\delta(\zeta)$ includes a constant $C=C(d,k)\geq 1$. We summarize their methods to determine $\gamma$ and $C$ in \Cref{algorithmkappa}. Afterward, we can compute the exact value of $\gamma=\gamma(d,k)$ in \Cref{main thm} using \Cref{rmk: kappa} and \Cref{algorithmkappa}.

Before explaining why \Cref{algorithmkappa} works, we specify the type of Laurent polynomials for which we can determine the necessary size of $\delta(\zeta)$.

\begin{defn}\label{c admissible}
    Let $P\in\overline{\Q}[T_1^{\pm 1},\ldots,T_d^{\pm d}]$ be a non-zero Laurent polynomial, $\zeta\in\mathbb{G}_m^d$ be a torsion point and $c\geq 1$. Then we say $(P,\zeta)$ is $c$-admissible if the following holds.

    Given any $A\in\operatorname{GL}_d(\Z)$ and an integer $\ell$ with $0\leq\ell\leq d-1$ suppose that
    \begin{itemize}
        \item $\zeta^A=(\eta,\xi)$ for $\eta\in\mathbb{G}_m^\ell$ and $\xi\in\mathbb{G}_m^{d-\ell}$,
        \item $P_{A,\eta}(T_1,\ldots,T_{d-\ell})\coloneqq Q(\eta,T_1,\ldots,T_{d-\ell})\in\overline{\Q}[T_1,\ldots,T_{d-\ell}]$, where
        \[
        Q(T_1,\ldots,T_d)\coloneqq P\big( (T_1,\ldots,T_d)^{A^{-1}} \big)\cdot R(T_1,\ldots,T_d)
        \]
        with $P((T_1,\ldots,T_d)^{A^{-1}})$ a Laurent polynomial and $R(T_1,\ldots,T_d)$ a monomial such that $Q$ is a polynomial coprime to $T_1\cdots T_d$. 
        \item Under the convention $\operatorname{inf}\varnothing=\infty$, 
        \[
        \rho(u)\coloneqq\operatorname{inf}\left\{ |v|:v\in\Z^{d-\ell}\backslash\{0\},\ \langle u,v\rangle =0 \right\},\quad\text{for } u\in\Z^{d-\ell},
        \]
        and
        \begin{align*}
            B(P_{A,\eta})\coloneqq\operatorname{inf}\big\{ B\in\N_{>0}:&\text{ let }\omega\in\mathbb{G}_m^{d-\ell}\text{ be a torsion point, }z\in S^1\text{ be a non-torsion point}\\
            &\text{ and }u\in\Z^{d-\ell}.\text{ If }P_{A,\eta}(\omega\cdot z^u)=0,\text{ then }\rho(u)\leq B \big\}.
        \end{align*}
    \end{itemize}
    Then we have $P_{A,\eta}\neq 0$ and $B(P_{A,\eta})\leq c|A^{-1}|$.
\end{defn}

We refer the reader to \cite[Section 6]{DimitrovHabegger} for more details regarding the value $B(P_{A,\eta})$. When $B(P_{A,\eta})$ is smaller, we can obtain better properties of $P_{A,\eta}$, and consequently, of $P$. One of the conditions of \Cref{DH main thm} is that $P$ is an essentially atoral Laurent polynomial. By \cite[Lemma 8.1]{DimitrovHabegger}, in this case, there exists $c\geq 1$ such that $(P,\zeta)$ is $c$-admissible for all torsion points $\zeta\in\mathbb{G}_m^d$ with $\delta(\zeta)\geq c$. 
The following proposition is derived by \cite[Theorem 8.8]{DimitrovHabegger}.

\begin{prop}\label{the size of delta zeta}
    Let $P\in\overline{\Q}[T_1,\ldots,T_d]\backslash\{0\}$ be a polynomial with at most $k$ nonzero terms for an integer $k\geq 2$. Then there exists a constant $C=C(d,k)$ such that the following holds.
    
    Let $\zeta\in\mathbb{G}_m^d$ be a torsion point such that $P(\zeta^\sigma)\neq 0$ for all $\sigma\in\operatorname{Gal}(\Q(\zeta)/\Q)$. If $(P,\zeta^\sigma)$ is $c$-admissible for all $\sigma\in\operatorname{Gal}(\Q(\zeta)/\Q)$, then the estimation of the convergence speed of the equidistribution in \Cref{DH main thm} holds for
    \[
    \delta(\zeta)\geq C\cdot\operatorname{max}\{c,\operatorname{deg}P\}^C.
    \]
    Moreover, the constant $C$ can be computed by \Cref{algorithmkappa}.
\end{prop}

\begin{ex}
    Suppose $P(X,Y)=X-1$. By \Cref{P i are essentially atoral}, the polynomial $P$ is essentially atoral as it is a binomial. We claim that $(P,\zeta)$ is $1$-admissible for $\zeta\neq(1,\xi)$, where $\xi$ is a torsion point in  $\mathbb{G}_m$. Indeed, suppose
    $A=\begin{pmatrix}
        a & b\\
        c & d
    \end{pmatrix}
    \in\operatorname{GL}_2(\Z)$. Then its determinant and inverse are
    \[
    \operatorname{det}A=ad-bc=\pm 1\quad\text{and}\quad
    A^{-1}=\frac{1}{ad-bc}\begin{pmatrix}
        d & -b\\
        -c & a
    \end{pmatrix}.
    \]
    Suppose $\zeta=(e^{\mathrm{i}2\pi x}, e^{\mathrm{i}2\pi y})$ for $(x,y)\in[0,1)^2$. Then $\zeta^A=(e^{\mathrm{i}2\pi (ax+cy)},e^{\mathrm{i}2\pi (bx+dy)})$.

    Let us first consider the case $ad-bc=1$. If we choose $\ell=0$ in \Cref{c admissible}, then
    \[
    P((X,Y)^{A^{-1}})=P(X^dY^{-c},X^{-b}Y^a)=X^dY^{-c}-1\quad\text{and}\quad P_{A,\eta}(X,Y)=X^d-Y^c.
    \]
    Since $ad-bc=1$, we have $d\neq 0$ or $c\neq 0$ and thus $P_{A,\eta}\neq 0$. Let us compute $B(P_{A,\eta})$. Let~ $\omega=(\omega_1,\omega_2)\in\mathbb{G}_m^2$ be a torsion point, $z\in S^1$ be a non-torsion point and $u=(u_1,u_2)\in\Z^2\backslash\{0\}$. Then $z^u=(z^{u_1},z^{u_2})$ and
    \[
    P_{A,\eta}(\omega z^u)=P(\omega_1z^{u_1},\omega_2z^{u_2})=\omega_1^dz^{u_1d}-\omega_2^cz^{u_2c}.
    \]
    If $P_{A,\eta}(\omega z^u)=0$, then $z^{u_2c-u_1d}=\omega_1^d\omega_2^{-c}$. As $\omega$ is a torsion point, there exists a positive integer~ $n$ such that $z^{n(u_2c-u_1d)}=1$. Since $z$ is not a torsion point, we have $n(u_2c-u_1d)=0$ and so
    \[
    u_2c-u_1d=0.
    \]
    Hence, if $d=0$ or $u_2=0$, then $u=(u_1,0)$ and so $\rho(u)=|(0,1)|=1$. Otherwise $d\neq 0$ and~ $u_2\neq 0$. Since $ad-bc=1$, we have $\operatorname{gcd}(c,d)=1$. For each $v=(v_1,v_2)\in\Z^2\backslash\{0\}$ with $\langle u,v\rangle=0$, we get
    \[
    v_2=-\frac{u_1}{u_2}v_1=-\frac{c}{d}v_1.
    \]
    Thus, $v_1=nd$ and $v_2=-nc$ for some $n\in\Z\backslash\{0\}$. It follows that $\rho(u)=\operatorname{max}\{|c|,|d|\}$ and so~ $B(P_{A,\eta})=\operatorname{max}\{1,|c|,|d|\}\leq |A^{-1}|$.

    We now turn to the case $\ell=1$ in \Cref{c admissible}. In this case, $\eta=e^{\mathrm{i}2\pi (ax+cy)}$ and
    \[
    P_{A,\eta}(X)=e^{\mathrm{i}2\pi d(ax+cy)}-X^c.
    \]
    If $P_{A,\eta}=0$, then $c=0$ and $e^{\mathrm{i}2\pi d(ax+cy)}=1$, which means $da x\in\Z$. Since $ad-bc=1$, then $ad=1$ and thus $x\in\Z$. Therefore, $\zeta=(1,\xi)$ for a torsion point $\xi\in\mathbb{G}_m$ if $P_{A,\eta}=0$. Now we exclude the case that $P_{A,\eta}=0$. Let us compute $B(P_{A,\eta})$. Let $\omega\in\mathbb{G}_m$ be a torsion point, $z\in S^1$ be a non-torsion point and $u\in\Z$ such that $P_{A,\eta}(\omega z^u)=0$. Then
    \[
    z^{uc}=\omega^{-c}e^{\mathrm{i}2\pi d(ax+cy)}.
    \]
    Similarly to the case where $\ell=0$, we obtain $uc$=0. If $c=0$, then $P_{A,\eta}$ is a constant. When~ $P_{A,\eta}\neq 0$, we have $B(P_{A,\eta})=1$. Otherwise, $u=0$, and thus $\rho(u)=1$. Therefore, $B(P_{A,\eta})=1$. The case $ad-bc=-1$ is essentially the same. Hence, the claim is proved.

    By similar arguments, one can show that
    \begin{itemize}
        \item the polynomial $P(X,Y)=Y-1$ is essentially atoral and $(P,\zeta)$ is $1$-admissible for $\zeta\neq(\xi,1)$, where $\xi\in\mathbb{G}_m$ is a torsion point,
        \item the Laurent polynomial $P(X,Y)=\frac{X}{Y}-1$ is essentially atoral and $(P,\zeta)$ is $2$-admissible for $\zeta\neq(\xi,\xi)$, where $\xi\in\mathbb{G}_m$ is a torsion point.
    \end{itemize}
\end{ex}

We will now explain why \Cref{algorithmkappa} is effective.

In the proof of \cite[Theorem 8.8]{DimitrovHabegger}, several estimates are involved that ignore constants with respect to $d$. We make all the necessary estimates explicit through the following observations. Recall that for $A=(a_{ij})\in\operatorname{GL}_d(\Z)$, we set $|A|=\operatorname{max}_{1\leq i,j\leq d}|a_{ij}|$. Then for any two $A,B\in\operatorname{GL}_d(\Z)$, we have $|AB|\leq d|A||B|$.

\begin{lem}\label{maximum norm_inverse of A}
    Let $A\in\operatorname{GL}_d(\Z)$. Then $|A^{-1}|\leq d^{2d-2}|A|^{d-1}$.
\end{lem}
\begin{proof}
    One can express $A^{-1}$ in terms of $\operatorname{det}A$, the powers of $A$ and their traces through the Cayley-Hamilton method, see for example \cite[Appendix B]{inverseA},
    \[
    A^{-1}=\frac{1}{\operatorname{det}(A)}\sum_{s=0}^{d-1}A^s\sum_{k_1,\ldots,k_{d-1}}\prod_{\ell=1}^{d-1}\frac{(-1)^{k_\ell+1}}{\ell^{k_\ell}k_\ell!}\left(\operatorname{tr}(A^\ell)\right)^{k_\ell},
    \]
    where $\operatorname{tr}(A^\ell)$ is the trace of $A^\ell$, the sum is taken over $s$ and all integers $k_\ell\geq 0$ satisfying the equation $s+\sum_{\ell=1}^{d-1}\ell k_\ell=d-1$. Hence,
    \[
    |A^{-1}|\leq\sum_{s=0}^{d-1}|A^s|\left|\sum_{k_1,\ldots,k_{d-1}}\prod_{\ell=1}^{d-1}\frac{(-1)^{k_\ell+1}}{\ell^{k_\ell}k_\ell!}\left(\operatorname{tr}(A^\ell)\right)^{k_\ell}\right|\leq\sum_{s=0}^{d-1} d^{s-1}|A|^s\left|\sum_{k_1,\ldots,k_{d-1}}\prod_{\ell=1}^{d-1}\frac{(-1)^{k_\ell+1}}{\ell^{k_\ell}k_\ell!}\left(\operatorname{tr}(A^\ell)\right)^{k_\ell}\right|.
    \]
    Note that
    \[
    \left|\operatorname{tr}(A^\ell)\right|\leq d|A^\ell|\leq d^\ell|A|^\ell.
    \]
    Since $(k_1,\ldots,k_{d-1})$ has at most $d^{d-1}$ possible choices, we have
    \[
    \left|\sum_{k_1,\ldots,k_{d-1}}\prod_{\ell=1}^{d-1}\frac{(-1)^{k_\ell+1}}{\ell^{k_\ell}k_\ell!}\left(\operatorname{tr}(A^\ell)\right)^{k_\ell}\right|\leq d^{d-1}\prod_{\ell=1}^{d-1}\left| \operatorname{tr}(A^\ell) \right|^{k_\ell}=d^{d-1}(d|A|)^{\sum_{\ell=1}^{d-1}\ell k_\ell}=d^{2d-s-2}|A|^{d-1-s}.
    \]
    Therefore,
    \[
    |A^{-1}|\leq\sum_{s=0}^{d-1} d^{2d-3}|A|^{d-1}\leq d^{2d-2}|A|^{d-1}.
    \]
\end{proof}

Suppose $k\geq 2$. The main way to prove \cite[Theorem 8.8]{DimitrovHabegger} is by induction on $d$. First of all, the univariate case was proven in \cite[Section 4]{DimitrovHabegger} by a repulsion property of the unit circle. For $d\geq 2$, a case with more restrictions on $P$ was shown in \cite[Proposition 6.2]{DimitrovHabegger}. 

In order to use the induction hypothesis, in \cite[Lemma 8.7]{DimitrovHabegger}, Dimitrov and Habegger chose a parameter $\epsilon\in(0,\frac{1}{2}]$, found an integer $\ell\in\{0,1,\ldots,d-1\}$ and a matrix $V\in\operatorname{GL}_d(\Z)$, considered the point $\zeta^V=(\eta,\xi)$ with $\eta\in\mathbb{G}_m^\ell$ and $\xi\in\mathbb{G}_m^{d-\ell}$, and constructed a polynomial~ $P_{V,\eta}\in\overline{\Q}[T_1,\ldots,T_{d-\ell}]$ with less variables satisfying the conditions in \cite[Proposition 6.2]{DimitrovHabegger}. The discrete sum in \Cref{DH main thm} with respect to $(P,\zeta)$ is given by the discrete sum of $(P_{V,\eta},\xi)$ under some Galois actions, see \cite[(8.8)]{DimitrovHabegger}. They then constructed a polynomial $\widetilde{Q}\in\overline{\Q}[T_1,\ldots,T_\ell]$ that is related to $P(X^{V^{-1}})$ and, consequently, associated with $P_{V,\eta}$. This construction was discussed in \cite[Section 7.2]{DimitrovHabegger}. The key point is that in this way we can associate $m(P_{V,\eta})$ with $m(P(X^{V^{-1}}))$ and thus with $m(P)$. Finally, we can apply the induction hypothesis to $(\widetilde{Q},\eta)$.

The construction of $P_{V,\eta}$ depends on $V$ and $\eta$, and some properties of $P_{V,\eta}$ in the proof of \cite[Theorem 8.8]{DimitrovHabegger} are implied by the upper bound of $|V^{-1}|$. Therefore, we explain how to obtain an explicit estimate of $|V^{-1}|$ following the proof of \cite[Lemma 8.7]{DimitrovHabegger}. We denote the identity matrix in $GL_d(\Z)$ by $I$. In that proof, the matrix $V$ is constructed inductively by a sequence of matrices $V_0\coloneqq I,V_1,\ldots,V_{\ell-1},V\coloneqq V_\ell\in\operatorname{GL}_d(\Z)$ and each $V_t$ is obtained by $V_{t-1}$ and a matrix $V_t'\in\operatorname{GL}_{d+1-t}(\Z)$ with
\[
|V_t|\leq d|V_{t-1}||V_t'|.
\]
The construction of the matrix $V_t'$ using the following lemma.

\begin{lem}\label{lem: construction of Vt'}
    Let $a\in\Z^d\backslash\{0\}$ be primitive. Then there exists a matrix $A\in\operatorname{GL}_d(\Z)$ whose first column is $a$ and $|A|\leq 2^{ \operatorname{max}\{0,d-2\} }|a|$.
\end{lem}
\begin{proof}
    Let $e_1,e_2,\ldots,e_d$ be the standard basis of $\Z^d$. Without loss of generality, we suppose that the first coordinate of $a$ is nonzero. Then $a,e_2,\ldots,e_d$ is linearly independent and a basis of a sublattice of $\Z^d$. By \cite[Theorem I.B]{GeoOfNumbersJWS} and \cite[Corollary 1]{GeoOfNumbersJWS}, we can find a basis~ $a_1,\ldots,a_d$ of $\Z^d$ such that
    \begin{align*}
        a&=\alpha_{11}a_1,\\
        e_2&=\alpha_{21}a_1+\alpha_{22}a_2,\\
        \cdots\\
        e_d&=\alpha_{d1}a_1+\alpha_{d2}a_2+\cdots+\alpha_{dd}a_d,
    \end{align*}
    where $\alpha_{ij}$ are integers for all $i,j$; $\alpha_{ii}>0$ and $0\leq\alpha_{ij}<\alpha_{ii}$ for all $i,j$ with $j<i$. Since $a$ is primitive, we have $|a_1|=|a|$. Moreover,
    \[
    |a_2|=\left| \frac{1}{\alpha_{22}}e_2 -\frac{\alpha_{21}}{\alpha_{22}}a_1 \right|\leq \frac{\alpha_{21}+1}{\alpha_{22}}|a|\leq |a|.
    \]
    We claim that $|a_i|\leq 2^{i-2}|a|$ if $i\geq 2$. Indeed, for each $k$ we have
    \begin{align*}
        |a_k|&=\left| \frac{1}{\alpha_{kk}}e_{k}-\frac{\alpha_{k1}}{\alpha_{kk}}a_1- \frac{\alpha_{k2}}{\alpha_{kk}}a_2-\cdots-\frac{\alpha_{k(k-1)}}{\alpha_{kk}}a_{k-1} \right|\\
        &\leq\frac{1+(\alpha_{kk}-1)|a_1|+(\alpha_{kk}-1)|a_2|+\cdots+(\alpha_{kk}-1)|a_{k-1}|}{\alpha_{kk}}\\
        &\leq|a_1|+|a_2|+\cdots+|a_{k-1}|.
    \end{align*}
    Therefore,
    \[
    |a_i|\leq |a_1|+\cdots+|a_{i-2}|+|a_{i-1}|\leq 2(|a_1|+\cdots+|a_{i-2}|)\leq\cdots\leq 2^{i-2}|a_1|=2^{i-2}|a|.
    \]
    Let $A$ be the matrix whose columns are $a_1=a,\ a_2,\ \ldots,\ a_d$. The lemma is proved.
\end{proof}

Applying \Cref{lem: construction of Vt'}, we get a matrix $V_t'\in\operatorname{GL}_{d+1-t}(\Z)$ whose first column is a primitive vector $v$ with $|v|\leq\delta(\zeta)^{\epsilon^{d-t}}$ given in the proof of \cite[Lemma 8.7]{DimitrovHabegger} such that
\[
|V_t'|\leq 2^{d-1}|v|\leq 2^{d-1}\delta(\zeta)^{\epsilon^{d-t}}
\]
and thus
\[
|V_t|\leq 2^{d-1} d|V_{t-1}|\delta(\zeta)^{\epsilon^{d-t}}
\]
for every $0\leq t\leq \ell$. Therefore,
\begin{equation}\label{maximum norm_V}
    |V|=|V_\ell|\leq d^\ell 2^{(d-1)\ell} \delta(\zeta)^{\epsilon^{d-1}+\epsilon^{d-2}+\cdots+\epsilon^{d-\ell}}\leq d^{d \ell} \delta(\zeta)^{2\epsilon^{d-\ell}}.
\end{equation}
By \Cref{maximum norm_inverse of A}, we have
\begin{equation}\label{ineq:V inverse}
    |V^{-1}|\leq d^{2d-2}|V|^{d-1}\leq d^{2d-2+(d-1)d\ell}\delta(\zeta)^{2\epsilon^{d-\ell}(d-1)}.
\end{equation}

In addition, an explicit estimate of $\operatorname{deg}P_{V,\eta}$ is required. The definition of $P_{V,\eta}$ is given in \Cref{c admissible} and one can verify that if $P\in\overline{\Q}[T_1,\ldots,T_d]$, then
\begin{equation}\label{ineq:deg P eta}
    \operatorname{deg}P_{V,\eta}\leq 2d|V^{-1}|\operatorname{deg}P\leq 2d^{2d-1+(d-1)d\ell}\delta(\zeta)^{2\epsilon^{d-\ell}(d-1)}\operatorname{deg}P\leq d^{2d+(d-1)d\ell}\delta(\zeta)^{2\epsilon^{d-\ell}(d-1)}\operatorname{deg}P.
\end{equation}

To apply the induction hypothesis to $(\widetilde{Q},\eta)$, we need to know how large $\delta(\eta)$ is. By the definition of strictness degree, there exists an $a\in\Z^\ell\backslash\{0\}$ with $|a|=\delta(\eta)$ such that $\eta^a=1$. Therefore,
\[
\zeta^{V\cdot(a,0)^T}=(\eta,\xi)^{(a,0)}=\eta^a=1
\]
and so
\[
\delta(\zeta)\leq\left| V \begin{bmatrix}
    a\\
    0
\end{bmatrix} \right|\leq d|V||a|= d|V|\delta(\eta).
\]
If $\epsilon^{d-\ell}\leq\frac{1}{4}$, then by (\ref{maximum norm_V}),
\begin{equation}\label{ineq:delta eta}
    \delta(\eta)\geq\frac{\delta(\zeta)}{d|V|}\geq\frac{\delta(\zeta)^{1-2\epsilon^{d-\ell}}}{d^{d\ell+1}}\geq\frac{\delta(\zeta)^{\frac{1}{2}}}{d^{d\ell+1}}.
\end{equation}

We can now read through the proof of \cite[Theorem 8.8]{DimitrovHabegger} and gather all the required inequalities as follows. The sequence of inequalities presented here corresponds precisely to the order in which they appeared in the proof, and all notations are consistent with those used in the proof, except that here we denote $\gamma$ as $\kappa$.

We take parameters $v_1,\ldots,v_{d-1}\in\left(0,\frac{1}{128 d^2}\right]$, $\epsilon\in\left(0,\frac{1}{2}\right]$ and $v_d=\frac{1}{128d^2}$ satisfying the following inequalities.

\begin{enumerate}
    \item In the case $\ell=0$, we need to choose $\epsilon$ and $v_1$ such that
    \[
    \delta(\zeta)^{\operatorname{min}\{\epsilon^{d-1},v_1^d/2\}}>d^{\frac{1}{2}}\operatorname{max}\{c,\operatorname{deg}P\}.
    \]
    By the condition in \cite[Theorem 8.8]{DimitrovHabegger}, we have
    \[
    \delta(\zeta)^{\operatorname{min}\{\epsilon^{d-1},v_1^d/2\}}\geq C^{\operatorname{min}\{\epsilon^{d-1},v_1^d/2\}}\operatorname{max}\{c,\operatorname{deg}P\}^{C\operatorname{min}\{\epsilon^{d-1},v_1^d/2\}}.
    \]
    Therefore, we can require that
    \begin{equation}\label{ineqs:1}
        C^{\operatorname{min}\{\epsilon^{d-1},v_1^d/2\}}>d^{\frac{1}{2}}\quad\text{and}\quad C\cdot\operatorname{min}\left\{\epsilon^{d-1},\frac{v_1^d}{2}\right\}>1.
    \end{equation}

    \item In the case $\ell=0$ in the proof, by applying \cite[Proposition 6.2]{DimitrovHabegger} we get
    \begin{equation}\label{ineqs:2}
        \gamma\leq\operatorname{min}\left\{\frac{v_1^d}{20d},\frac{\epsilon^{d-1}}{16(k-1)},\frac{v_1^d}{32(k-1)}\right\}.
    \end{equation}

    \item For all $\ell\leq d-2$,
    \begin{equation}\label{ineqs:3}
        \epsilon\leq\frac{v_{\ell+1}^d}{4}.
    \end{equation}

    \item For all $\ell\leq d-2$,
    \begin{equation}\label{ineqs:4}
        \epsilon^{d-\ell-1}-2\epsilon^{d-\ell}d\geq\frac{\epsilon^{d-\ell-1}}{2}.
    \end{equation}

    \item Consider Step 1 and the case $d-\ell\geq 2$ in the proof. By the inequality (\ref{ineq:V inverse}), we have
    \[
    B(P_{V,\eta})\leq c d^{2d-1+(d-1)d\ell}\delta(\zeta)^{2\epsilon^{d-\ell}(d-1)}\leq c d^{2d+(d-1)d\ell}\delta(\zeta)^{2\epsilon^{d-\ell}d}.
    \]
    We need to verify that
    \[
    \widetilde{\lambda}(\xi;v_{\ell+1})>(d-\ell)^{\frac{1}{2}}\operatorname{max}\left\{ B(P_{V,\eta}),\operatorname{deg}P_{V,\eta} \right\}.
    \]
    It is shown in the proof that $\widetilde{\lambda}(\xi;v_{\ell+1})\geq\delta(\zeta)^{\epsilon^{d-\ell-1}}$. By the inequality (\ref{ineq:deg P eta}), it is sufficient to verify that
    \[
    \delta(\zeta)^{\epsilon^{d-\ell-1}}\geq d^{\frac{1}{2}}d^{2d+(d-1)d\ell}\delta(\zeta)^{2\epsilon^{d-\ell}d}\operatorname{max}\{c,\operatorname{deg}P\}.
    \]
    By the inequality (\ref{ineqs:4}), it is implied by
    \[
    \delta(\zeta)^{\epsilon^{d-\ell-1}/2}\geq d^{2d+\frac{1}{2}+(d-1)d\ell}\operatorname{max}\{c,\operatorname{deg}P\}.
    \]
    Thus we can require that
    \begin{equation}\label{ineqs:5}
        C^{\epsilon^{d-\ell-1}/2}\geq d^{2d+\frac{1}{2}+(d-1)d\ell}\quad\text{and}\quad C\cdot\frac{\epsilon^{d-\ell-1}}{2}\geq 1.
    \end{equation}

    \item For all $1\leq\ell\leq d-1$,
    \begin{equation}\label{ineqs:6}
        \epsilon^{d-\ell}\leq\frac{1}{4}.
    \end{equation}

    \item For all $1\leq\ell\leq d-1$,
    \begin{equation}\label{ineqs:7}
        5(v_1+v_2+\cdots+v_\ell)+4\epsilon^{d-\ell}d\leq\frac{v_{\ell+1}^d}{80d}.
    \end{equation}

    \item For all $1\leq\ell\leq d-1$,
    \begin{equation}\label{ineqs:8}
        32\epsilon^{d-\ell}d^3\leq\frac{\epsilon^{d-\ell-1}}{32k}.
    \end{equation}

    \item For all $1\leq\ell\leq d-1$,
    \begin{equation}\label{ineqs:9}
        \gamma\leq\operatorname{min}\left\{ \frac{v_{\ell+1}^d}{80},\frac{\epsilon^{d-\ell-1}}{32k} \right\}.
    \end{equation}

    \item For all $1\leq\ell\leq d-1$,
    \begin{equation}\label{ineqs:10}
        1-4\epsilon^{d-\ell}d\cdot C(\ell,k^2)\geq\frac{1}{2}.
    \end{equation}

    \item Consider the part in Step 2 in the proof taking values of $C$. By \cite[Lemma 8.3]{DimitrovHabegger}, we have $\big( Q,(\eta^\sigma,\xi^\sigma) \big)$ is $cd|V^{-1}|$-admissible, and hence $cd^{2d+(d-1)d\ell}\delta(\zeta)^{2\epsilon^{d-\ell}d}$-admissible by (\ref{ineq:V inverse}). Then \cite[Lemma 8.4]{DimitrovHabegger} implies that $(\widetilde{Q},\eta^\sigma)$ is $cd^{2d+1+(d-1)d\ell}\delta(\zeta)^{2\epsilon^{d-\ell}d}$-admissible. By the inequality (\ref{ineq:deg P eta}),
    \[
    \operatorname{deg}Q\leq d^{2d+(d-1)d\ell}\delta(\zeta)^{2\epsilon^{d-\ell}(d-1)}\operatorname{deg}P.
    \]
    Therefore, we can apply \cite[Lemma 7.2]{DimitrovHabegger} to deduce that
    \[
    \operatorname{deg}\widetilde{Q}\leq(\ell+1)\operatorname{deg}Q\leq d^{2d+1+(d-1)d\ell}\delta(\zeta)^{2\epsilon^{d-\ell}d}\operatorname{deg}P.
    \]
    We need to verify that
    \[
    \delta(\eta)\geq C(\ell,k^2)\cdot\operatorname{max}\left\{ cd^{2d+1+(d-1)d\ell}\delta(\zeta)^{2\epsilon^{d-\ell}d},\operatorname{deg}\widetilde{Q} \right\}^{C(\ell,k^2)},
    \]
    which is implied by
    \[
    \delta(\eta)\geq C(\ell,k^2) d^{(2d+1+(d-1)d\ell)C(\ell,k^2)} \delta(\zeta)^{2\epsilon^{d-\ell}d\cdot C(\ell,k^2)}\operatorname{max}\{c,\operatorname{deg}P\}^{C(\ell,k^2)}.
    \]
    By the inequality (\ref{ineq:delta eta}) and (\ref{ineqs:10}), it is sufficient to verify that
    \[
    \delta(\zeta)^{\frac{1}{4}}\geq C(\ell,k^2) d^{(2d+1+(d-1)d\ell)C(\ell,k^2)+d\ell+1}\operatorname{max}\{c,\operatorname{deg}P\}^{C(\ell,k^2)}.
    \]
    Hence, we can require that
    \begin{equation}\label{ineqs:11}
        C^{\frac{1}{4}}\geq C(\ell,k^2) d^{(2d+1+(d-1)d\ell)C(\ell,k^2)+d\ell+1}\quad\text{and}\quad\frac{C}{4}\geq C(\ell,k^2)
    \end{equation}
    for all $1\leq\ell\leq d-1$.

    \item For all $1\leq\ell\leq d-1$,
    \begin{equation}\label{ineqs:12}
        \frac{\gamma(\ell,k^2)}{4}\geq 32\epsilon^{d-\ell}d^3.
    \end{equation}

    \item For all $1\leq\ell\leq d-1$,
    \begin{equation}\label{ineqs:13}
        \frac{1}{96(d+1)k^2}-2\epsilon^{d-\ell}d\geq\frac{1}{100(d+1)k^2}.
    \end{equation}

    \item For all $1\leq\ell\leq d-1$,
    \begin{equation}\label{ineqs:14}
        \gamma\leq\operatorname{min}\left\{ \frac{1}{100(d+1)k^2},\frac{\gamma(\ell,k^2)}{4} \right\}.
    \end{equation}
\end{enumerate}

We will now simplify these inequalities presented above. The following are the simplified inequalities related to $\epsilon$.

\begin{itemize}
    \item The inequality (\ref{ineqs:4}) is equivalent to $\epsilon\leq\frac{1}{4d}$, which is a consequence of (\ref{ineqs:3}). Additionally, the inequality (\ref{ineqs:6}) is also implied by (\ref{ineqs:3}).

    \item The inequality (\ref{ineqs:8}) is equivalent to $\epsilon\leq 1/(2^{10}d^3k)$.

    \item For the inequality (\ref{ineqs:10}), we can require that $\epsilon\leq 1/(2^3d\cdot C(\ell,k^2))$ for all $1\leq\ell\leq d-1$.

    \item For the inequality (\ref{ineqs:12}), we can require that $\epsilon\leq \gamma(\ell,k^2)/(2^7d^3)$ for all $1\leq\ell\leq d-1$.

    \item The inequality (\ref{ineqs:13}) is equivalent to
    \[
    \epsilon^{d-\ell}\leq\frac{1}{2^6\cdot 3 \cdot 5^2 d(d+1)k^2},\quad\forall 1\leq\ell\leq d-1,
    \]
    which simplifies to
    \begin{equation}\label{ineq:choose of epsilon 2}
        \epsilon\leq 1/(2^6\cdot 3 \cdot 5^2 d(d+1)k^2).
    \end{equation}

    \item We will try to find a specific way to choose $v_1,\ldots,v_{d-1}$ that satisfies (\ref{ineqs:7}). If we require that
    \begin{equation}\label{ineq:choose of epsilon}
        \epsilon\leq\frac{v_\ell^d}{4d\cdot 160d},\quad\forall 1\leq\ell\leq d
    \end{equation}
    and $v_1\leq v_2\leq\cdots\leq v_d$,
    then for (\ref{ineqs:7}) it suffices to check that
    \[
    5\ell v_\ell+\frac{ v_{\ell+1}^d v_{\ell+2}^d\cdots v_d^d }{ (4d)^{d-\ell-1}(160d)^{d-\ell} }\leq\frac{v_{\ell+1}^d}{80d}.
    \]
    Thus, we can choose \[
    v_\ell=\frac{v_{\ell+1}^d}{5\ell\cdot 80d}\left( 1-\frac{v_{\ell+2}^d\cdots v_d^d}{ 2(4d\cdot 160d)^{d-\ell-1} } \right),
    \]
    where $\ell$ goes from $d-1$ to $1$. Note that (\ref{ineqs:3}) is implied by (\ref{ineq:choose of epsilon}).
\end{itemize}

In the following, we simplify the inequalities related to $\gamma$.

\begin{itemize}
    \item Combining all these inequalities (\ref{ineqs:2}), (\ref{ineqs:9}) and (\ref{ineqs:14}) for $\gamma$ and $v_1\leq v_2\leq\cdots v_d$, we obtain
    \begin{align*}
        \gamma\leq\operatorname{min}\Big\{ \frac{v_1^d}{20d},\frac{\epsilon^{d-1}}{16(k-1)},\frac{v_1^d}{32(k-1)},\frac{v_2^d}{80},\frac{\epsilon^{d-2}}{32k},\frac{1}{100(d+1)k^2},\frac{\gamma(1,k^2)}{4},\cdots,\frac{\gamma(d-1,k^2)}{4}\Big\}.
    \end{align*}
    Since $k\geq 2$, we have $\frac{\epsilon^{d-1}}{16(k-1)}\leq\frac{\epsilon^{d-2}}{32k}$. From the inequality (\ref{ineq:choose of epsilon}), we also find that $\frac{v_1^d}{20d},\ \frac{v_1^d}{32(k-1)},\ \frac{v_2^d}{80}$ are less than $\frac{\epsilon^{d-1}}{16(k-1)}$. It follows from (\ref{ineqs:12}) that $\frac{\epsilon^{d-1}}{16(k-1)}\leq\frac{\gamma(\ell,k^2)}{4}$ for all $\ell$. By the inequality (\ref{ineq:choose of epsilon 2}), we have $\frac{\epsilon^{d-1}}{16(k-1)}\leq\frac{1}{100(d+1)k^2}$. In summary, it suffices to choose
    \[
    \gamma(d,k)=\frac{\epsilon^{d-1}}{16(k-1)}.
    \]
\end{itemize}

The following are the simplified inequalities related to $C$.

\begin{itemize}
    \item By the inequality (\ref{ineqs:3}), the inequality (\ref{ineqs:1}) becomes
    \[
    C^{\epsilon^{d-1}}>d^{\frac{1}{2}}\quad\text{and}\quad \epsilon^{d-1}C>1.
    \]

    \item To make sure that (\ref{ineqs:5}) holds for all $1\leq\ell\leq d-2$, we can require that
    \[
    C^{\frac{\epsilon^{d-2}}{2}}\geq d^{d^3}\quad\text{and}\quad C\geq\frac{2}{\epsilon^{d-2}}.
    \]

    \item To make sure that (\ref{ineqs:11}) holds, we can require that for all $1\leq\ell\leq d-1$,
    \[
    C^{\frac{1}{4}}\geq C(\ell,k^2)d^{(d^3+2)C(\ell,k^2)}\quad\text{and}\quad C\geq 4\cdot C(\ell,k^2).
    \]
\end{itemize}

We can now summarize all these arguments to present \Cref{algorithmkappa}. 

\begin{rmk}
    This algorithm may be far from optimal; for instance, to satisfy (\ref{ineqs:7}), we choose relatively small values for $\epsilon$ and $v_i$ in (\ref{ineq:choose of epsilon}).

    Additionally, our choice of $\gamma$ depends on the second term in the Vigogradov's notation in \cite[Proposition 6.2(ii)]{DimitrovHabegger}. This term is provided by \cite[Theorem A.1]{DimitrovHabegger}, which states that $m(P(T_1,\ldots,T_d)^u)$ converges to $m(P)$ as $\rho(u)\to\infty$ and gives an upper bound on the rate of this convergence. Specifically, this term represents the convergence speed. The convergence speed is also studied in \cite{Mahler1} and \cite{Mahler2}, where it is shown to be much better than that given in \cite[Theorem A.1]{DimitrovHabegger} for certain polynomials $P$ and vectors $u$. However, a complete description of this convergence speed remains an open question. 
\end{rmk}

\begin{algorithm}
    \caption{Computing $\gamma(d,k)$ and $C(d,k)$}\label{algorithmkappa}
    \begin{algorithmic}[1]
        \Require{ $d\geq 2,\ k\geq 2,$ \\
        $C(1,k)=C(1,k^2)=\cdots=C(1,k^{2^{d-1}})=1,$ \\
        $\gamma(1,k)=\gamma(1,k^2)=\cdots=\gamma(1,k^{2^{d-1}})=1-\epsilon_0$ for an $\epsilon_0>0$.}
        \Ensure{ $\gamma(d,k)$, $C(d,k)$ }
        \For{$n=2$ to $d-1$}
            \For{$m=1$ to $d-n$}
                \State{$v_n\gets\frac{1}{128 n^2},$ $v_{n-1}\gets 1,\ \ldots,\ v_1\gets 1,\ s\gets n-1$}
                \While{ $s\geq 1$ }
                    \State{$v_s\gets\frac{v_{s+1}^n}{5s\cdot 80n}\left( 1-\frac{v_{s+2}^n\cdots v_n^n}{2(4n\cdot 160n)^{n-s-1}} \right)$ }
                    \State{$s\gets s-1$}
                \EndWhile
                \State{$\epsilon_1\gets\operatorname{min}\left\{ \frac{1}{2^3n\cdot C(1,k^{2^{m+1}})},\ldots,\frac{1}{2^3n\cdot C(n-1,k^{2^{m+1}})} \right\}$}
                \State{$\epsilon_2\gets\operatorname{min}\left\{ \frac{\gamma(1,k^{2^{m+1}})}{2^7n^3},\ldots,\frac{\gamma(n-1,k^{2^{m+1}})}{2^7n^3} \right\}$}
                \State{$\epsilon\gets\operatorname{min}\left\{ \frac{v_1^n}{2^6\cdot 5 n^2},\frac{1}{2^{10}\cdot n^3k^{2^m}},\frac{1}{2^7\cdot 3\cdot 5^2 n(n+1)k^{2^{m+1}}},\epsilon_1,\epsilon_2 \right\}$}
                \State{$\gamma(n,k^{2^m}) \gets \frac{\epsilon^{n-1}}{16(k^{2^{m}}-1)}$}
                \State{$C_1\gets\operatorname{min}\left\{ 
C(1,k^{2^{m+1}})^4 n^{4(n^3+2)C(1,k^{2^{m+1}})},\ldots,C(n-1,k^{2^{m+1}})^4 n^{4(n^3+2)C(n-1,k^{2^{m+1}})} \right\}$}
                \State{$C_2\gets\operatorname{min}\left\{ 
4\cdot C(1,k^{2^{m+1}}),\ldots,4\cdot C(n-1,k^{2^{m+1}}) \right\}$}
                \State{$C(n,k^{2^m})\gets\operatorname{min}\left\{ n^{\frac{1}{2\epsilon^{n-1}}}+1, \frac{1}{\epsilon^{n-1}}+1, n^{\frac{2n^3}{\epsilon^{n-2}}},\frac{2}{\epsilon^{n-2},}, C_1,C_2\right\}  $}
            \EndFor
        \EndFor
        \State{ $v_d\gets\frac{1}{128d^2},\ v_{d-1}\gets 1,\ \ldots,\ v_1\gets 1,\ \ell\gets d-1$ }
        \While{$\ell\geq 1$}
            \State{$v_\ell=\frac{v_{\ell+1}^d}{5\ell\cdot 80d}\left( 1-\frac{v_{\ell+2}^d\cdots v_d^d}{2(4d\cdot 160d)^{d-\ell-1}} \right)$}
            \State{$\ell\gets \ell-1$}
        \EndWhile
        \State{$\epsilon\gets\operatorname{min}\left\{ \frac{v_1^d}{2^7\cdot 5 d^2},\frac{1}{2^{10}\cdot d^3k},\frac{1}{2^6\cdot 3\cdot 5^2 d(d+1)k^2},\frac{1}{2^3d\cdot C(1,k^2)},\ldots,\frac{1}{2^3d\cdot C(d-1,k^2)},\frac{\gamma(1,k^2)}{2^7d^3},\ldots,\frac{\gamma(d-1,k^2)}{2^7d^3} \right\}$}
        \State{$\gamma(d,k)\gets \frac{\epsilon^{d-1}}{16(k-1)}$}
        \State{$C_1\gets\operatorname{min}\left\{  C(1,k^2)^4 d^{4(d^3+2)C(1,k^2)},\ldots,C(d-1,k^2)^4 d^{4(d^3+2)C(d-1,k^2)}\right\}$}
        \State{$C_2\gets\operatorname{min}\left\{ 
4\cdot C(1,k^2),\ldots,4\cdot C(d-1,k^2) \right\}$}
        \State{$C(d,k)\gets\operatorname{min}\left\{ d^{\frac{1}{2\epsilon^{d-1}}}+1, \frac{1}{\epsilon^{d-1}}+1, d^{\frac{2d^3}{\epsilon^{d-2}}},\frac{2}{\epsilon^{d-2},},C_1,C_2\right\}$}
    \end{algorithmic}
\end{algorithm}

\newpage
\bibliographystyle{alpha}
\bibliography{ref.bib}

\end{document}